\documentclass{amsart}

\usepackage[utf8]{inputenc}
\usepackage[usenames,dvipsnames]{xcolor}
\usepackage{hyperref}
\usepackage{amssymb}
\usepackage{float}
\usepackage{tikz}
\usepackage{forest}
\usepackage{nicefrac}
\usepackage{enumerate}
\usepackage[nameinlink, capitalize, noabbrev]{cleveref}
\usetikzlibrary{shapes.geometric,decorations.markings}
\usetikzlibrary{decorations.pathreplacing}
\usetikzlibrary{fit}
\usetikzlibrary{positioning}
\usetikzlibrary{intersections}
\tikzset{mytext/.style={font=\small, text=black}}

\hypersetup{
    colorlinks=true,
    linkcolor=teal,
    citecolor=magenta,
    }

\newtheorem{Theorem}{Theorem}

\newtheorem{Conjecture}[Theorem]{Conjecture}
\newtheorem{Corollary}[Theorem]{Corollary}
\newtheorem{proposition}{Proposition}[section]
\newtheorem{lemma}[proposition]{Lemma}
\newtheorem{corollary}[proposition]{Corollary}
\newtheorem{theorem}[proposition]{Theorem}

\newtheorem{Question}[Theorem]{Question}
\theoremstyle{definition}
\newtheorem{remark}[proposition]{Remark}
\newtheorem{definition}[proposition]{Definition}

\numberwithin{equation}{section}

\title[First-order theory, rigidity and Boston's conjecture]{Weakly branch actions: first-order theory, rigidity and Boston's conjecture}
\author{Jorge Fariña-Asategui}
\address{Jorge Fariña-Asategui: Centre for Mathematical Sciences, Lund University, 223 62 Lund, Sweden -- Department of Mathematics, University of the Basque Country UPV/EHU, 48080 Bilbao, Spain}
\email{jorge.farina\_asategui@math.lu.se}
\thanks{The author is supported by the Spanish Government, grant PID2020-117281GB-I00, partly with FEDER funds. The author also acknowledges support from the Walter Gyllenberg Foundation from the Royal Physiographic Society of Lund}

\begin{document}

\begin{abstract}

    We disprove a well-known conjecture of Boston (2000), which claims that a just-infinite pro-$p$ group is branch if and only if it admits a positive-dimensional embedding in the group of $p$-adic automorphisms. This is obtained as a result of a comprehensive study of the rigidity of branch actions.

    Firstly, we generalize the notion of the structure graph, introduced by Wilson in 2000, to weakly branch groups and use it to prove several results on the first-order theory of weakly branch groups, extending previous results of Wilson on branch groups.
    
    Secondly, we completely characterize the rigidity of weakly branch and branch actions on arbitrary spherically homogeneous rooted trees, extending previous partial results (for branch actions) by Hardy, Garrido, Grigorchuk and Wilson. Moreover, we prove that rigidity of a weakly branch group is equivalent to rigidity of its closure in the full automorphism group.

    Thirdly, we extend greatly the sufficient conditions $(*)$ and $(**)$ of Grigorchuk and Wilson, which leads to a complete and very easy-to-check characterization of the rigidity of the weakly branch actions of a fractal group of $p$-adic automorphisms. We further establish the first connection in the literature between the Hausdorff dimension of a weakly branch action and its rigidity.
    
    Lastly, we put everything together to show that the zero-dimensional just-infinite branch pro-$p$ groups introduced recently by the author admit rigid branch actions on a tree obtained by deletion of levels. This, together with previous results of the author, shows that these groups are indeed counterexamples to the aforementioned conjecture of Boston.
\end{abstract}

\maketitle

\section{introduction}
\label{section: introduction}

Let $T$ be a spherically homogeneous rooted tree and let $\mathrm{Aut}~T$ be the group of automorphisms of $T$ fixing the root of $T$. Let us fix a subgroup $G\le \mathrm{Aut}~T$. We say that~$G$ is \textit{level-transitive} if $G$ acts transitively on every level of $T$, where vertices at each level $n$ of $T$ are precisely those at distance $n$ from the root. Given a vertex $v\in T$, we define the \textit{rigid vertex stabilizer} $\mathrm{rist}_G(v)$, as the subgroup of $G$ consisting of those automorphisms in $G$ fixing every vertex outside the subtree~$T_v$ rooted at~$v$. The so-called \textit{rigid level stabilizer} $\mathrm{Rist}_G(n)$ is no more than the direct product of the rigid vertex stabilizers corresponding to all the vertices at the $n$th level of~$T$. We say that $G$ is branch (or weakly branch) if each $\mathrm{Rist}_G(n)$ is of finite index in~$G$ (respectively infinite). 

Recall that as a profinite group $\mathrm{Aut}~T$ admits a metric space structure induced by the filtration $\{\mathrm{St}(n)\}_{n\ge 0}$, where we write $\mathrm{St}(n)$ for the $n$th \textit{level stabilizer}, i.e. the subgroup of $\mathrm{Aut}~T$ stabilizing pointwisely every vertex at the $n$th level of $T$. We may define a Hausdorff dimension $\mathrm{hdim}_T(\cdot)$ for the Borel subsets of $\mathrm{Aut}~T$. If $G\le \mathrm{Aut}~T$ is a closed subgroup, then its Hausdorff dimension coincides with its lower-box dimension by \cite[Theorem 2.4]{BarneaShalev}, i.e. it is given by
\begin{align*}
    \mathrm{hdim}_T(G)=\liminf_{n\to\infty}\frac{\log|G:\mathrm{St}_G(n)|}{\log|\mathrm{Aut}~T:\mathrm{St}(n)|},
\end{align*}
where we write $\mathrm{St}_G(n):=G\cap \mathrm{St}(n)$. For $T_p$ the $p$-adic tree, i.e. the regular rooted tree of prime degree $p$, we write $W_p$ for the Sylow pro-$p$ subgroup of $\mathrm{Aut}~T_p$ and call it the group of \textit{$p$-adic automorphisms}; see \cref{section: p-adic} for an alternative definition of $W_p$. For a closed subgroup $G\le W_p$, we usually consider its Hausdorff dimension in $W_p$ instead, for which we write $\mathrm{hdim}_{W_p}(G)$.

Branch groups were introduced by Grigorchuk in 1997 as a class of groups generalizing several earlier constructions of groups acting on rooted trees. The first example of a branch group, the so-called \textit{first Grigorchuk group}, was introduced by Grigorchuk in \cite{GrigorchukBurnside} as a group with remarkable properties: it is a Burnside group (i.e. a finitely generated infinite torsion group), it was the first example of a group with intermediate growth \cite{GrigorchukMilnor}, answering an open problem of Milnor \cite{Milnor}, and it was also the first example of an amenable but not elementary amenable group, answering another open problem of Day \cite{Day}. The class of weakly branch groups also includes groups with notable properties; such as the Basilica group, defined by Grigorchuk and $\dot{\mathrm{Z}}$uk in \cite{Basilica}, which was the first example of an amenable group of exponential growth \cite{BasilicaAmenable}. Remarkable applications of these classes of groups outside group theory include, the celebrated solution by Bartholdi and Nekrashevych to Hubbard's twisted rabbit problem in complex dynamics \cite{BartholdiNekra} and applications to different density problems in arithmetic dynamics; cf. \cite{Bridy, JorgeSantiFPP,JonesComp,Juul,Odoni, Radi}.

Branch groups are conjectured to play an important role in arithmetic geometry \cite{NewHorizonsBoston}. By a well-known conjecture of Fontaine and Mazur \cite{FontaineMazur}, the Galois groups associated to maximal unramified pro-$p$ extensions of a number field are expected to not admit infinite $p$-adic analytic quotients, i.e. their linear representations over the $p$-adic integers should have finite image; see \cite{Calegari,Kisin,Pan,SkinnerWiles} for some progress regarding this central conjecture in arithmetic geometry. Moreover, Boston further conjectured in \cite{Boston}, that the same holds for linear representations over any pro-$p$ domain. 

The Fontaine-Mazur conjecture is strongly connected to the Hausdorff dimension of groups acting on the $p$-adic tree. It was conjectured by Abért and Virág in \cite[Conjecture 8.4]{AbertVirag} that embeddings of linear groups over a pro-$p$ domain into~$W_p$ are zero-dimensional, a conjecture which has been recently confirmed by the author in \cite[Theorem~A]{ArborealJorge}. Furthermore, the author showed in \cite[Corollary 3]{ArborealJorge}, that positive-dimensional just-infinite pro-$p$ subgroups of $W_p$ satisfy Boston's generalization of the Fontaine-Mazur conjecture; where a profinite group is \textit{just-infinite} if all its non-trivial closed normal subgroups are open. Similarly, a discrete group is said to be \textit{just-infinite} if all its proper quotients are finite.

As proved by Grigorchuk in \cite{NewHorizonsGrigorchuk} (based on an earlier classsification of Wilson in \cite{Wilson}), just-infinite branch groups constitute one of the three (two in the profinite setting) classes of groups partitioning the class of just-infinite groups. Just-infinite pro-$p$ groups appear naturally in arithmetic geometry as they constitute the critical cases of the Fontaine-Mazur conjecture; see Boston's discussion in \cite[Section 3]{NewHorizonsBoston}.

The Fontaine-Mazur conjecture suggests that classical $p$-adic Galois representations are not the appropriate object to study the maximal unramified pro-$p$ extensions of a number field. Boston suggested to consider representations on the $p$-adic tree $T_p$ instead. In particular, motivated by the classification of just-infinite pro-$p$ groups, he conjectured that the just-infinite pro-$p$ quotients of these Galois groups can be indeed realized as branch subgroups of $W_p$. This fact, together with the evidence that most branch groups have positive Hausdorff dimension in $W_p$, led Boston to pose the following purely (geometric) group theoretic conjecture in \cite[Problem 9]{NewHorizonsBoston}:

\begin{Conjecture}[Boston's conjecture, 2000]
\label{Conjecture: Boston conjecture}
    A just-infinite pro-$p$ group is branch if and only if it admits an embedding into $W_p$ with positive Hausdorff dimension.
\end{Conjecture}

Boston's conjecture has been open for 25 years. What is more, the class of branch groups first appeared in print in \cite{NewHorizonsGrigorchuk}, i.e. in the same book where Boston posed \cref{Conjecture: Boston conjecture}. Boston's conjecture is thus one of the first fundamental conjectures regarding the structure of branch groups, appearing right at the birth of this class of groups. Boston's conjecture was already considered by Abért and Virág in their seminal paper \cite{AbertVirag}, but, to the best of our knowledge, there has not been any further progress regarding its veracity until the recent breakthrough in \cite{RestrictedSpectra}, where the very first examples of just-infinite branch pro-$p$ groups with zero Hausdorff dimension in~$W_p$ were constructed by the author. These examples provide strong evidence against the veracity of Boston's conjecture. 

However, the examples in \cite{RestrictedSpectra} are far from being counterexamples to Boston's conjecture. In fact, one first needs to address the following critical question: are the other embeddings of these zero-dimensional just-infinite branch groups into $W_p$ also zero-dimensional? Since non-branch just-infinite pro-$p$ subgroups of $W_p$ are known to be zero-dimensional, we need to understand the embeddings of these groups into~$W_p$ as branch groups. The goal of the present paper is to develop the tools needed to answer this critical question and ultimately disprove Boston's conjecture.

Let $G$ be an abstract group. An embedding $\rho:G\to\mathrm{Aut}~T$ such that $\rho(G)$ is a (weakly) branch group is called a \textit{(weakly) branch action} of $G$. In principle, the same group $G$ can have completely different (weakly) branch actions on different spherically homogeneous rooted trees. However, by the work of Rubin in \cite{Rubin}, the action of $G$ (as a weakly branch group) on the boundary $\partial T$ is unique, up to conjugation in $\mathrm{Homeo}~\partial T$; see also \cite{NekraLav}. A natural question is whether a weakly branch action of $G$ on $T$ is also unique up to conjugation in $\mathrm{Aut}~T$. If the answer is positive, we say that $G$ is \textit{$T$-rigid}.

The work done in \cite{RestrictedSpectra} has brought to the forefront the importance of understanding branch actions and their rigidity. Rigidity of branch actions seems to be more subtle than its counterpart on the boundary of the tree, and it is less understood. We shall close this gap and give a thorough study of branch actions and their rigidity. What is more, as weakly branch actions are precisely those actions on the tree that induce Rubin actions on the boundary of the tree (i.e. rigid actions), we shall work in the more general framework of weakly branch actions. The first step in this paper is hence to generalize the tools available to study branch actions to the context of weakly branch actions.

The main tools to study the branch actions of a group date back to 1971, when Wilson introduced the \textit{structure lattice} of a just-infinite group in \cite{Wilson}, i.e. the Boolean algebra consisting of equivalence classes of subnormal subgroups up to an equivalence relation in terms of their centralizers \cite{Wilson}. Wilson introduced the structure lattice in order to classify just-infinite groups; see also \cite{NewHorizonsGrigorchuk}. A detailed description of the structure lattice of a (non-necessarily just-infinite) branch group has been obtained recently by the author and Grigorchuk in \cite{JorgeSlava}. 

The subgraph of the structure lattice of a branch group $G$ restricted to basal subgroups (where a \textit{basal} subgroup is a subgroup whose normal closure is a direct product of finitely many of its distinct conjugates) is called the \textit{structure graph} of~$G$. The structure graph of a branch group~$G$ is the main tool to study its branch actions, as it encodes every possible branch action of~$G$ on any spherically homogeneous rooted tree; see \cite{AlejandraCSP}.

The goal of \cref{section: the basal graph of a weakly branch group} is to extend the definition of the structure graph, to encode weakly branch actions. The difficulties to do so were already highlighted by Garrido in her doctoral thesis \cite[pages 75-76]{GarridoPhD}. Namely, there is no nice class of groups $\mathcal{C}$ such that all the proper quotients of a weakly branch group are virtually $\mathcal{C}$. This contrasts with branch groups, whose proper 
quotients are all virtually abelian, a key fact in the definition of the structure lattice.  Therefore, we opt for a different approach here. Instead of generalizing the structure lattice to weakly branch groups and only then obtain the structure graph by restricting to basal subgroups, we generalize directly the structure graph. In fact, for a weakly branch group $G$, we set two basal subgroups to be equivalent if and only they have non-trivial intersection and equal normalizers. The resulting graph $\mathcal{B}_G$ of these equivalence classes will be called the \textit{structure graph} of $G$, as it
does indeed generalize the structure graph of a branch group; see \cref{section: the basal graph of a weakly branch group}.

This allows us to extend results of Wilson on the first-order theory of branch groups  in \cite{Wilson2} to the context of weakly branch groups in \cref{section: first order theory}. Recall that the \textit{congruence topology} in $G$ is the metric topology induced by the level-stabilizers.  We summarize the results obtained in \cref{section: first order theory} in the following theorem:

\begin{Theorem}
\label{Theorem: definability}
    Let $G\le \mathrm{Aut}~T$ be a weakly branch group. Then, the following are all first-order definable in $G$:
    \begin{enumerate}[\normalfont(i)]
        \item the congruence topology in $G$;
        \item the structure graph $\mathcal{B}_G$.
    \end{enumerate}
    Furthermore, if $G$ is $T$-rigid then
    \begin{enumerate}[\normalfont(i)]
        \item[\normalfont(iii)] each level-stabilizer $\mathrm{St}_G(n)$ is also first-order definable in $G$.
    \end{enumerate}
\end{Theorem}

We note that a group-theoretic definition of the congruence topology in a branch group and a weakly branch group was given by Garrido in \cite{AlejandraCSP} and in her doctoral thesis \cite[Theorem 6.5]{GarridoPhD} respectively. The first-order definability in (i) and (iii) of \cref{Theorem: definability} are new even for branch groups. In fact, (iii) extends a result of Nekrashevych and Lavreniuk stating that $\mathrm{St}_G(n)$ is characteristic if $G$ is a level-transitive iterated wreath product \cite[Theorem 8.2]{NekraLav}.

In \cref{section: rigidity}, we begin a systematic study of weakly branch actions. Our main result in this section is a complete characterization of the rigidity of a weakly branch action on an arbitrary spherically homogeneous rooted tree $T$. This extends previous partial results of Grigorchuk and Wilson (for branch groups) in \cite{GrigorchukWilson}, of Garrido and Hardy (for branch groups) in their respective doctoral thesis \cite{GarridoPhD, Hardy}, and of Wilson (also for branch groups) in \cite{WilsonBook}. It also unifies the approach of Wilson to study branch actions via the structure graph and the alternative approach of Nekrasevych and Lavreniuk in \cite{NekraLav} to study automorphisms of weakly branch groups, as $T$-rigidity of a weakly branch group $G\le \mathrm{Aut}~T$ implies that its automorphism group coincides with its normalizer in $\mathrm{Aut}~T$. The connection between the ideas in \cite{NekraLav} and the approach of Wilson was explicitly asked by Garrido in \cite[Remark 3 in Section 7.1]{GarridoPhD}.

We characterize $T$-rigidity in terms of $T$-filtrations, a new notion that we introduce in \cref{section: rigidity}; see \cref{definition: T filtration}. Let us a fix a group $G$ and a weakly branch action $\rho:G\to \mathrm{Aut}~T$. For each $v\in T$, we write $\mathrm{st}_\rho(v)$ for the subgroup of $G$ corresponding, via $\rho$, to the stabilizer of $v$ in~$\rho(G)$. We also write
$$\mathcal{S}_\rho:=\{\mathrm{st}_\rho(v)\}_{v\in T}.$$
Intuitively, a $T$-filtration in the group $G$ is a filtration of finite-index subgroups of~$G$, which resembles the filtration $\{\mathrm{st}_\rho(v)\}_{v\in \gamma}$ for an infinite rooted path $\gamma\in \partial T$. We write~$\mathcal{F}_{T}$ for the collection of all subgroups of~$G$ which belong to some $T$-filtration in $G$.

\begin{Theorem}
\label{proposition: characterization of rigidity}
    Let $G$ be a group admitting a (weakly) branch action $\tau:G\to\mathrm{Aut}~T$. The following are all equivalent:
    \begin{enumerate}[\normalfont(i)]
        \item the group $G$ is $T$-rigid;
        \item for every (weakly) branch action $\rho:G\to \mathrm{Aut}~T$, we have $\mathcal{F}_T=\mathcal{S}_{\rho}$;
        \item there exists a (weakly) branch action $\rho:G\to \mathrm{Aut}~T$, such that $\mathcal{F}_T=\mathcal{S}_{\rho}$.
    \end{enumerate}
\end{Theorem}

\cref{proposition: characterization of rigidity} is the first result on the rigidity of a (weakly) branch group where the structure graph is not required to be isomorphic to the tree $T$. In fact,
\cref{proposition: characterization of rigidity} shows that $T$-rigidity is equivalent to the structure graph containing a unique rooted copy of~$T$.

Let $G$ be a discrete weakly branch group. Note that by \cref{Theorem: definability} (see also \cite[Theorem 6.5]{GarridoPhD}), the congruence topology in $G$ is independent of the weakly branch action, so the closure of the image of any weakly branch action of $G$ in $\mathrm{Aut}~T$ is isomorphic to the completion $\overline{G}$ of $G$ with respect to the congruence topology. Thus, a weakly branch action $\tau:G\to\mathrm{Aut}~T$ induces a weakly branch action $\overline{\tau}:\overline{G}\to\mathrm{Aut}~T$. Therefore, one may wonder to what extent are the $T$-rigidity of $G$ and $\overline{G}$ related. We prove the following:

\begin{Theorem}
    \label{Theorem: T-rigidity in finite quotients}
    Let $G$ be a group admitting a weakly branch action $\tau:G\to\mathrm{Aut}~T$. Then $G$ is $T$-rigid if and only if its completion $\overline{G}$ with respect to the congruence topology is $T$-rigid. 
\end{Theorem}

\cref{Theorem: T-rigidity in finite quotients} is quite surprising to us. To illustrate this fact, note that a recent result of Adams and Hyde in~\cite[Theorem 1.1]{AdamsHyde} implies that the closure of the iterated monodromy group of a unicritical Thurston polynomial only depends on its post-critical dynamics, i.e. on the orbit of its unique critical point. Thus, \cref{Theorem: T-rigidity in finite quotients} implies, for example, that the iterated monodromy group of the rabbit polynomial is $T_2$-rigid if and only if the iterated monodromy groups of the corabbit and the airplane polynomials are both $T_2$-rigid \cite{Thurston}, or that the iterated monodromy group of $z^2+i$ is $T_2$-rigid if and only if the Grigorchuk group
given by the
sequence $(01)^\infty$ (which corresponds to an obstructed Thurston polynomial) is $T_2$-rigid \cite{Thurston}. In particular, the rigidity of the iterated monodromy groups of unicritical Thurston polynomials seems to not be affected by Thurston obstructions; see \cite{Thurston, DouadyHubbard} for further details on Thurston obstructions.

In \cref{section: p-adic}, we specialize to weakly branch actions of fractal groups of $p$-adic automorphisms and give a characterization for $T_p$-rigidity which is very easy to check in this context. Grigorchuk and Wilson introduced the following sufficient conditions $(*)$ and $(**)$ for rigidity in \cite{GrigorchukWilson}. Note that we say that two distinct vertices $u,u'\in T$ are incomparable if none is a descendant of the other.

\begin{enumerate}
    \item[$(*)$] For each vertex $v\in T$ the stabilizer $\mathrm{st}_G(v)$ acts as a (transitive) cyclic group of prime order on the immediate descendants of $v$;
    \item[$(**)$] whenever $u$ and $u'$ are incomparable vertices and $v$ is a descendant of (not equal to) $u$, there is an element $g\in G$ fixing $u'$ and moving $v$.
\end{enumerate}

To the best of our knowledge, conditions $(*)$ and $(**)$ are still the state of the art for proving rigidity of branch actions; see for instance \cite{MoritzGGS} (note that Hardy's property (S) in \cite{Hardy} is not easy to check, it is a sufficient condition more along the lines of the characterization in \cref{proposition: characterization of rigidity}, while Garrido's conditions (A) and (B) in \cite{GarridoPhD} are just a slight generalization of conditions $(*)$ and $(**)$). We greatly generalize condition $(**)$ in \cref{theorem: sufficient condition for rigidity in p-adic}. 

Let $T_d$ be the $d$-regular rooted tree. We may identify $T_d$ with the free monoid of rank $d$ and write its vertices as finite words.  For $v\in T_d$ and $g\in \mathrm{Aut}~T_d$,  the \textit{section} $g|_v\in \mathrm{Aut}~T_d$ of $g$ at $v$ is the unique automorphism satisfying that, for every $u\in T_d$, we have
$$(vu)^g=v^gu^{g|_v}.$$

Recall that a group $G\le \mathrm{Aut}~T_d$ such that $g|_v\in G$ for every $v\in T_d$ and every $g\in G$ is called \textit{self-similar}. Furthermore, a self-similar group $G\le \mathrm{Aut}~T_d$ is \textit{fractal}, if $G$ is level-transitive and for every $v\in T_d$, the projection $\varphi_v:\mathrm{st}_G(v)\to G$ given by $g\mapsto g|_v$ is surjective.

Note that if a fractal group $G\le \mathrm{Aut}~T_d$ satisfies condition $(*)$, then $d=p$ is prime and $G\le W_p$.

\begin{Theorem}
\label{Corollary: sufficient for p-adic}
    Let $G\le W_p$ be a fractal weakly branch group. Then $G$ is $T_p$-rigid if and only if
    $$\frac{G}{\mathrm{St}_G(2)}\not\cong C_p\times C_p.$$
\end{Theorem}

We note that the implication ``$G/\mathrm{St}_G(2)\cong C_p\times C_p \implies G$ is not $T_p$-rigid" holds for any (not necessarily fractal) weakly branch group $G\le W_p$. 

\cref{Corollary: sufficient for p-adic} makes it very easy to decide $T_p$-rigidity for fractal groups of $p$-adic automorphisms. It immediately yields $T_p$-rigidity of many well-known examples of branch and weakly branch groups: the first Grigorchuk group \cite{GrigorchukBurnside}, all the GGS-groups acting on the $p$-adic tree \cite{GGSHausdorff} and their generalizations \cite{AnithaJone}, the Basilica group \cite{Basilica} and its different generalizations to the $p$-adic tree \cite{pBasilica, GeneralizedBasilica}, and the Brunner-Sidki-Vieira group and its generalizations \cite{BSV,BSV2}, among others.

In particular, \cref{Corollary: sufficient for p-adic} implies that if $G\le W_p$ is a fractal weakly branch group containing a level-transitive element, then $G$ is $T_p$-rigid. This is the case for the iterated monodromy groups of unicritical Thurston polynomials. Indeed, these groups are fractal and they are given by kneading automata, so the product of all the generators in any order is always a level-transitive element \cite{SelfSimilar}. Thus, such an iterated monodromy group is $T_p$-rigid whenever it is weakly branch. For example, this is the case for the quadratic Thurston polynomials not linearly conjugate to the powering map $z^2$ or the Chebyshev polynomial $2z^2-1$ by a result of Bartholdi and Nekrashevych \cite[Theorems 3.12 and 4.10]{BartholdiNekra}, and for the closure of unicritical Thurston polynomials of degree $p$ not linearly conjugate to the powering map $z^p$ or the $p$th Chebyshev polynomial, by combining a result on the Hausdorff dimension of these groups due to Adams and Hyde \cite[Proposition~1.4]{AdamsHyde} and a result of the author \cite[Lemma~5.1]{AV}. Furthermore, this implies that a weakly branch action of the iterated monodromy group of a unicritical Thurston polynomial $f$, can only be the iterated monodromy action of $f$. In this case, the monodromy action of $f$ is given by the first-order theory of its iterated monodromy group by \cref{Theorem: definability}. This yields a new interplay between complex dynamics and model theory.

Weak branchness is closely related to Hausdorff dimension. If a self-similar closed subgroup $G$ has positive Hausdorff dimension in $\mathrm{Aut}~T$, then $G$ is weakly branch by the aforementioned result of the author \cite[Lemma~5.1]{AV}. More generally, an arbitrary positive-dimensional level-transitive closed subgroup of $\mathrm{Aut}~T$ is such that almost all of its projections are weakly branch \cite[Theorem~2.1]{ArborealJorge}.

An interesting consequence of \cref{Corollary: sufficient for p-adic} is that having large Hausdorff dimension in $W_p$ is an obstruction for fractal weakly branch subgroups of $W_p$ to not be $T_p$-rigid:

\begin{Corollary}
    \label{Corollary: hdim and rigidity}
    Let $G\le W_p$ be a fractal weakly branch group such that its closure $\overline{G}$ satisfies
    $$\mathrm{hdim}_{W_p}(\overline{G})>\frac{1}{p}.$$
    Then $G$ is $T_p$-rigid.
\end{Corollary}

To the best of the author's knowledge, \cref{Corollary: hdim and rigidity} is the first result in the literature relating the Hausdorff dimension of a weakly branch group to the rigidity of its weakly branch actions.

\cref{Corollary: hdim and rigidity} also allows us to use Hausdorff dimension as a group invariant. We can use this fact to show that some fractal weakly branch groups are not isomorphic as abstract groups. To illustrate this, we apply this criterion to show that the Basilica group and the Brunner-Sidki-Viera group \cite{BSV} are not isomorphic as abstract groups in \cref{corollary: basilica bsv}. 

We also note that the lower-bound in \cref{Corollary: hdim and rigidity} is sharp in general, as there exist non-rigid branch subgroups of $W_p$ with Hausdorff dimension precisely $1/p$. However, the lower bound in \cref{Corollary: hdim and rigidity} can be improved to include $1/p$ if the group $G$ is further assumed to be topologically finitely generated; see \cref{remark: tfg}.

Finally, in \cref{section: Boston conjecture}, we get back to Boston's conjecture. \cref{Corollary: hdim and rigidity} suggests that being zero-dimensional could be an obstruction for $T_p$-rigidity. Let us consider, for each $n\ge 1$, the just-infinite branch group $G_n\le W_p$ with zero-dimensional closure introduced by the author in \cite[Section 6]{RestrictedSpectra}; see \cref{section: Boston conjecture} for the definition of $G_n$. As $G_n$ satisfies that $G_n/\mathrm{St}_{G_n}(2)\cong C_p\times C_p$, it is not $T_p$-rigid as remarked below \cref{Corollary: sufficient for p-adic}. Thus, in order to disprove Boston's conjecture, we need to find a more subtle approach. 

There is a natural tree where each $G_n$ acts on: a tree $T_n$ obtained from the $p$-adic tree by deletion of levels; see \cref{section: Boston conjecture}. In fact, the proof of the zero-dimensionality of the closure of $G_n$ in $W_p$ \cite[Proposition~6.12]{RestrictedSpectra} does not use the action of $G_n$ on the $p$-adic tree, but its action on this (not regular) spherically homogeneous rooted tree $T_n$. Other groups acting naturally on these ``growing trees" have been recently considered by Skipper and Thillaisundaram in \cite{RachelAnitha}. Applying the same arguments as in \cite[Proposition~6.12]{RestrictedSpectra}, one only needs to show that the groups $G_n$ are $T_n$-rigid in order to obtain that all their possible branch actions on the $p$-adic tree are zero-dimensional. We prove $T_n$-rigidity of $G_n$ in \cref{theorem: my groups are T-rigid}. Hence, we finally obtain the following:

\begin{Theorem}
    \label{Theorem: Boston conjecture}
    For each $n\ge 1$, the closure in $W_p$ of the just-infinite branch group $G_n\le W_p$ is a counterexample to Boston's conjecture.
\end{Theorem}

\cref{Theorem: Boston conjecture} concludes the disproval of Boston's conjecture initiated by the author in \cite{RestrictedSpectra}.

We remark that in order to prove this $T_n$-rigidity, one needs the full power of the characterization in \cref{proposition: characterization of rigidity}. The previous results on rigidity of Hardy, Garrido, Grigorchuk and Wilson, do not apply to branch actions on the growing trees $T_n$. The proof of \cref{Theorem: Boston conjecture} suggests that it is always possible to find (with some work) a tree by deletion of levels, where the weakly branch actions of a given group are rigid. Thus, we pose the following question:

\begin{Question}
    Let $G\le \mathrm{Aut}~T$ be a weakly branch group. Does there always exist a tree $\widetilde{T}$ obtained by deletion of levels from $T$ such that $G$ is $\widetilde{T}$-rigid?
\end{Question}

To conclude, let us remark that even if \cref{Theorem: Boston conjecture} disproves Boston's conjecture, we believe that the intuition of Boston leading to \cref{Conjecture: Boston conjecture} was correct in the following sense. As we mentioned before, it has been proved in \cite[Corollary 3]{ArborealJorge} that positive-dimensional subgroups of $W_p$ satisfy Boston's generalization of the Fontaine-Mazur conjecture. Hence, it is reasonable to expect that the just-infinite pro-$p$ Galois groups of unramified pro-$p$ extensions of number fields may be realized as positive-dimensional subgroups of $W_p$, and thus as branch groups. This would establish the veracity of the Fontaine-Mazur conjecture and its generalization by Boston, as a direct application of \cite[Corollary 3]{ArborealJorge}.

\subsection*{\textit{\textmd{Notation}}} For a subset $S\subseteq G$, we write $\mathrm{N}_G(S)$ and $\mathrm{C}_G(S)$ for its normalizer and its centralizer in $G$ respectively. We also write  $\mathrm{C}_G^2(S):=\mathrm{C}_G(\mathrm{C}_G(S))$. We use exponential notation for the action of a group on a rooted tree and its boundary.  For subgroups $H\le L\le G$, we denote by $H^L\le G$ the normal closure of $H$ in $L$. For a (weakly) branch action $\rho:G\to\mathrm{Aut}~T$, we shall write $\mathrm{st}_\rho(v), \mathrm{rist}_\rho(v),\mathrm{St}_\rho(n)$ and $\mathrm{Rist}_\rho(n)$ for the subgroups of $G$ corresponding to the stabilizer and the rigid stabilizer of $v\in T$ in $\rho(G)$, and the stabilizer and rigid stabilizer of level $n\ge1$ in $\rho(G)$ respectively. We shall drop $\rho$ from the notation and simply write  $\mathrm{st}_G(v)$, etc to simplify notation when $G$ is viewed directly as a subgroup of $\mathrm{Aut}~T$, and $\mathrm{st}(v)$, etc when $G=\mathrm{Aut}~T$.

\subsection*{Acknowledgements} 

I would like to thank Gustavo A. Fernández-Alcober, Dominik Francoeur, Rostislav Grigorchuk and Santiago Radi for fruitful discussions on the structure lattice, the structure graph and the rigidity of branch actions, and Anitha Thillaisundaram for her useful feedback on a first version of this manuscript. I would like to also thank J. Moritz Petschick for pointing out the role of different branch actions in Boston's conjecture, which was the main motivation for the present paper.

\section{The structure graph of a weakly branch group}
\label{section: the basal graph of a weakly branch group}

Here, we introduce the main tool for the study of weakly branch actions: the structure graph. We define the structure graph from a new perspective, avoiding the difficulties to extend its definition to the class of weakly branch groups mentioned by Garrido in \cite{GarridoPhD}. The goal of the section is to show that the main properties of the structure graph of a branch group still hold in the context of weakly branch groups.

\subsection{Basal subgroups}

A subgroup $B\le G$ is called \textit{basal} if it has finitely many conjugates and its normal closure is the direct product of its distinct finitely many conjugates, i.e.
$$B^G=B\times B^{g_1}\times\dotsb\times B^{g_r}$$
for some $g_1,\dotsc,g_r\in G$. From the definition, it is immediate that for a basal subgroup $B\le G$ and any $g\in G$, the following are all equivalent:
\begin{enumerate}[\normalfont(i)]
    \item $B\ne B^g$;
    \item $[B,B^g]=1$;
    \item $B\cap B^g=1$.
\end{enumerate}

We write $B(G)$ for the collection of all basal subgroups of a group $G$. In the following lemma, we record some well-known properties of basal subgroups:

\begin{lemma}
\label{lemma: properties of basal subgroups}
    Let $B\in B(G)$ be a basal subgroup of a group $G$. Then the following holds:
    \begin{enumerate}[\normalfont(i)]
        \item the normalizer $\mathrm{N}_G(B)$ is of finite index in $G$;
        \item for any $L\ge \mathrm{N}_G(B)$, the subgroup $B^L$ is basal and $\mathrm{N}_G(B^L)=L$.
    \end{enumerate}
\end{lemma}
\begin{proof}
    Property (i) follows directly from the orbit-stabilizer theorem as the index of the normalizer of a subgroup $H\le G$ is the number of distinct conjugates of $H$ in $G$. Property (ii) is precisely \cite[Lemma 10.2.2(d)]{Hardy}.
\end{proof}

A canonical example of basal subgroups are rigid vertex stabilizers of a weakly branch group. In this case, for any $v\in T$, we have
$$\mathrm{N}_G(\mathrm{rist}_G(v))=\mathrm{st}_G(v),$$
which has finite index in $G$, and
$$\mathrm{rist}_G(v)^G=\prod_{w\text{ at level }n}\mathrm{rist}_G(w)=:\mathrm{Rist}_G(n),$$
where $n$ is the level at which $v$ lies.

\subsection{The structure graph of a branch group}

Wilson introduced the structure graph of a  just-infinite branch group $G$ in \cite{WilsonNewHorizons}, as the subgraph of the structure lattice restricted to basal subgroups.  We follow the more general presentation of Garrido in \cite{AlejandraCSP}. Let $G\le\mathrm{Aut}~T$ be a branch group. We write $H\le_{\mathrm{va}}G$ if $H$ contains the commutator subgroup of a finite-index subgroup of $G$. Let $L(G)$ be the collection of all subgroups of $G$ with finitely many conjugates. Note that $L(G)$ is a lattice with respect to subgroup inclusion and with the join and the meet of two subgroups $H,K\in L(G)$ given by $\langle H,K\rangle$ and $H\cap K$ respectively. By \cite[Lemma~2.5]{AlejandraCSP}, for $H, K\in L(G)$, the following are all equivalent:
\begin{enumerate}[\normalfont(i)]
    \item $H\cap K\le_{\mathrm{va}} H,K$;
    \item $\mathrm{C}_G(H)=\mathrm{C}_G(K)$;
    \item there exists $D\in L(G)$ such that $H\times D,K\times D\le_\mathrm{va} G$.
\end{enumerate}

Let $\sim_S $ be the equivalence relation in $L(G)$ given by any of the equivalent conditions (i)-(iii) above. As $\sim_S$ is a congruence, the quotient $L(G)/\sim_S$ admits a lattice structure, with join and meet induced by $\langle H,K\rangle$ and $H\cap K$ respectively. We call the lattice $L(G)/\sim_S$ the \textit{structure lattice} of $G$. 

The equivalence $\sim_S$ induces an equivalence relation in $B(G)$, which we shall denote also by $\sim_S$. Then, the \textit{structure graph} of $G$ is no more than the quotient $B(G)/\sim_S$. Note that the quotient $B(G)/\sim_S$ is a partially ordered set with respect to the partial order induced by inclusion of centralizers. We can endow $B(G)/\sim_S$ with a graph structure by joining $[B]$ and $[\widetilde{B}]$ with an edge if and only if $[B]\le [\widetilde{B}]$ (or $[\widetilde{B}]\le [B]$) and $[B]$ (respectively $[\widetilde{B}]$) is maximal among the equivalence classes strictly smaller than $[\widetilde{B}]$ (respectively $[B]$). Since for every $g\in G$ we have the equality $\mathrm{C}_G(B)^g=\mathrm{C}_G(B^g)$, the group $G$ acts naturally by conjugation on $B(G)/\sim_S$.

\subsection{An alternative definition}

Let  $G\le \mathrm{Aut}~T$ be now a weakly branch group. The equivalence of the properties (i)-(iii) above needs all the proper quotients of $G$ to be virtually abelian. Even if this is always the case for branch groups, it might not be the case for weakly branch groups; the Basilica group is indeed such an example \cite[Proposition 23]{pro-c}. Therefore, we consider a
different approach. We shall define directly a structure graph for $G$ by putting an appropriate equivalence relation $\sim$ directly in $B(G)$. The appropriate equivalence relation $\sim$  should satisfy the following:
\begin{enumerate}[\normalfont(a)]
    \item for any weakly branch action $\rho:G\to \mathrm{Aut}~T$, rigid vertex stabilizers should be distinguished by $\sim$, i.e. $\mathrm{rist}_\rho(v)\sim \mathrm{rist}_\rho(w)$ if and only if $v=w$;
    \item it should be compatible with the action of $G$ by conjugation on $B(G)$, i.e. $B\sim \widetilde{B}$ if and only if $B^g\sim \widetilde{B}^g$ for every $g\in G$;
    \item its equivalence classes should be as large as possible.
\end{enumerate}

Property (a) guarantees that all the weakly branch actions of $G$ are seen in $B(G)/\sim$, while property (b) reduces the study of all weakly branch actions of $G$ to a single action of $G$: its action by conjugation on $B(G)/\sim$. Property (c) guarantees that this action of $G$  on $B(G)/\sim$ by conjugation is as simple as possible. 

We define the equivalence relation $\sim$ in terms of normalizers as
$$B\sim \widetilde{B}\quad\text{if and only if} \quad B\cap \widetilde{B}\ne 1\text{ and }\mathrm{N}_G(B)=\mathrm{N}_G(\widetilde{B})$$
for every $B,\widetilde{B}\in B(G)$. Our motivation is twofold. 

On the one hand, if $G$ is branch, the equivalence relations $\sim$ and $\sim_S$ coincide in $B(G)$; compare \cite[Theorem 15.1.1]{Hardy} and \cref{lemma: description of basal graph with rists} below. Hence, for a branch group $G$, the quotient $\mathcal{B}_G:=B(G)/\sim$ is no more than the structure graph of Wilson. In general, for any weakly branch group $G$, the quotient $\mathcal{B}_G$ is a partially ordered set, where the partial order is induced by inclusion of normalizers of intersecting basal subgroups.  We shall call $\mathcal{B}_G$ the \textit{structure graph} of $G$, as it generalizes the structure graph of a branch group. The structure graph $\mathcal{B}_G$ has a natural graph structure defined in the same way as for branch groups. Note also that for every $g\in G$, we have $\mathrm{N}_G(B)^g=\mathrm{N}_G(B^g)$, so again $G$ acts naturally by conjugation on the structure graph $\mathcal{B}_G$ fulfilling (b).

On the other hand, there is various geometric reasons why we believe the equivalence relation $\sim$ is the best choice to study weakly branch actions (and even branch actions). First, as $G$ is acting on $B(G)$ via conjugation, normalizers of basal subgroups are precisely the stabilizers of this action. Since these stabilizers completely determine a weakly branch action, the equivalence relation $\sim$ is the equivalence relation satisfying (a) and (b) and with the largest possible equivalence classes; thus fulfilling~(c). Secondly, the equivalence relation $\sim$ yields almost directly intuitive ``geometric" tools, such as \cref{lemma: orbits of elements determine basal}, to study rigidity of weakly branch actions.

The structure graph $\mathcal{B}_G$ behaves very similarly for weakly branch groups as it does for branch groups. Therefore, many of its properties can be proved in the same way as for branch groups. However, there are some subtleties. The remainder of the section is devoted to establish these properties. Most of the results and their proofs are just adaptations of the analogous results for branch groups, so we shall skip many of the details and highlight how the equivalence relation $\sim$ replaces $\sim_S$ in the proofs.

\subsection{Properties of the structure graph}

As with branch actions, a weakly branch action $\rho:G\to \mathrm{Aut}~T$ induces a $G$-equivariant embedding of $T$ in $\mathcal{B}_G$ by property (a); see \cite[Proposition 3.1]{AlejandraCSP}:

\begin{lemma}
\label{lemma: actions as embeddings}
    For a weakly branch action $\rho:G\to \mathrm{Aut}~T$, we obtain a $G$-equivariant embedding $\phi_\rho:T\to \mathcal{B}_G$ via $v\mapsto [\mathrm{rist}_\rho(v)]$.
\end{lemma}
\begin{proof}
    The $G$-equivariance of $\phi_\rho$ follows from
    \begin{align*}
        v^{\rho(g)}\mapsto& [\mathrm{rist}_\rho(v^{\rho(g)})]=[\mathrm{rist}_\rho(v)^g]=[\mathrm{rist}_\rho(v)]^g.
    \end{align*}
    Let us prove the injectivity of $\phi_\rho$. If $v$ and $w$ are distinct vertices at the same level, then $\mathrm{rist}_\rho(v)\cap \mathrm{rist}_\rho(w)=1$, so their rigid vertex stabilizers correspond to different vertices in $\mathcal{B}_G$. If $v$ and $w$ are distinct vertices at different levels, then the normalizers of their rigid stabilizers, namely $\mathrm{st}_\rho(v)$ and $\mathrm{st}_\rho(v)$, have different indices in $G$, so they are distinct. Again, this implies that the rigid stabilizers of $v$ and $w$ correspond to different vertices in $\mathcal{B}_G$.
\end{proof}

\cref{lemma: actions as embeddings} shows that the structure graph $\mathcal{B}_G$ encodes all the possible weakly branch actions of the group $G$. For the remainder of the section, we shall fix a weakly branch action $\rho:G\to \mathrm{Aut}~T$ and identify $G$ with its image $\rho(G)$ to simplify notation. The following observation will be useful:

\begin{lemma}
\label{lemma: rist not virtually abelian}
    Let $G\le \mathrm{Aut}~T$ be a weakly branch group and $v\in T$. Then, for any finite-index subgroup $H\le G$, we have $(\mathrm{rist}_G(v)\cap H)'\ne 1$. 
\end{lemma}
\begin{proof}
    The subgroup $\mathrm{rist}_G(v)\cap H$ is of finite index in $\mathrm{rist}_G(v)$ and thus $\mathrm{rist}_G(v)\cap H$ is non-trivial, as $\mathrm{rist}_G(v)$ is infinite. Now, the same argument as in the proof of \cite[Lemma~2.17]{DominikMaximal} shows $(\mathrm{rist}_G(v)\cap H)'\ne 1$, as $\mathrm{rist}_G(w)\cap H\ne 1$ for any $w\in T$. 
\end{proof}

The following lemma can be proved similarly to \cite[Proposition 3.2]{AlejandraCSP}. Note that we do not need the rigid level stabilizers to be of finite index, it suffices to use \cref{lemma: rist not virtually abelian}.

\begin{lemma}
\label{lemma: every basal contains rist}
    Let $G$ be a weakly branch group. Let $1\ne B\in B(G)$ and $v\in T$ a vertex moved by $B$. Then $\mathrm{st}_G(v)\le \mathrm{N}_G(B)$ and $[\mathrm{rist}_G(v)]\le[B]$.
\end{lemma}
\begin{proof}
    Let $N$ denote the normal core of $\mathrm{N}_G(B)$. Then $N$ is non-trivial and of finite index in $G$, so $\mathrm{rist}_G(v)\cap N$ is non-trivial as $\mathrm{rist}_G(v)$ is infinite. Let $h,k\in \mathrm{rist}_G(v)\cap N$ and let $b\in B$ be an element such that $v^b\ne v$. Then 
    $$h^b\in \mathrm{rist}_G(v)^b=\mathrm{rist}_G(v^b)\ne \mathrm{rist}_G(v),$$
    and since rigid vertex stabilizers are basal $h^b$ must commute with $h$ and $k$. Then
    $$[[h^{-1},b],k]=[h,k].$$
    However, since $h,k\in \mathrm{N}_G(B)$ and $b\in B$, we get $[[h^{-1},b],k]\in B$ and thus $[h,k]\in B$. Hence $$(\mathrm{rist}_G(v)\cap N)'\le B.$$ 
    Then, for $g\in \mathrm{st}_G(v)$ we have $(\mathrm{rist}_G(v)\cap N)^g=(\mathrm{rist}_G(v)\cap N)$, and since the commutator subgroup is characteristic, we get
    $$(\mathrm{rist}_G(v)\cap N)'=\big((\mathrm{rist}_G(v)\cap N)'\big)^g\le B^g.$$
    Thus $B\cap B^g\ne 1$ as $(\mathrm{rist}_G(v)\cap N)'\ne 1$ by \cref{lemma: rist not virtually abelian}. Therefore $g\in \mathrm{N}_G(B)$, as~$B$ is basal. Hence $$\mathrm{N}_G(\mathrm{rist}_G(v))=\mathrm{st}_G(v)\le \mathrm{N}_{G}(B).$$
    In particular $[\mathrm{rist}_G(v)]\le[B]$, as $1 \ne (\mathrm{rist}_G(v)\cap N)' \le \mathrm{rist}_G(v) \cap B$.
\end{proof}

\begin{lemma}
\label{lemma: description of basal graph with rists}
    Let $1\ne B\in B(G)$. Then $[B]=[\mathrm{rist}_G(v)^{\mathrm{N}_G(B)}]$,
    for any vertex $v$ moved by $B$.
\end{lemma}
\begin{proof}
    By \cref{lemma: every basal contains rist}, we have $\mathrm{st}_G(v)\le \mathrm{N}_G(B)$ for any vertex $v$ moved by $B$. Now, the subgroup $\mathrm{rist}_G(v)^{\mathrm{N}_G(B)}$ is basal with normalizer $\mathrm{N}_G(B)$ by \cref{lemma: properties of basal subgroups}\textcolor{teal}{(ii)}. Furthermore $B\cap \mathrm{rist}_G(v)^{\mathrm{N}_G(B)}\ne 1$ by the proof of \cref{lemma: every basal contains rist}. Therefore $[B]=[\mathrm{rist}_G(v)^{\mathrm{N}_G(B)}]$.
\end{proof}

\section{The first-order theory of weakly branch groups}
\label{section: first order theory}

In this section, we review and adapt several results of Wilson in \cite{Wilson2} to the context of weakly branch groups and use them to prove \cref{Theorem: definability}.

\subsection{Uniqueness of the congruence topology}  First, note that a direct application of \cref{lemma: description of basal graph with rists} reproves a result of Garrido asserting that the congruence topologies arising from different weakly branch actions are all equivalent: 

\begin{lemma}[{see {\cite[Theorem 6.5]{GarridoPhD}}}]
\label{lemma: equivalent topologies}
    Let $G\le \mathrm{Aut}~T$ be a weakly branch group and let $\rho:G\to\mathrm{Aut}~T$ be another weakly branch action of $G$ on $T$. Then for every $n\ge 1$, there exists $k\ge n$ such that
    $$\mathrm{St}_\rho(n)\ge \mathrm{St}_G(k).$$
\end{lemma}
\begin{proof}
    Let  $n\ge 1$ and $v$ be any vertex at the $n$th level of $T$. Then
    $$[\mathrm{rist}_\rho(v)]=[\mathrm{rist}_G(u)^{\mathrm{N}_G(\mathrm{rist}_\rho(v))}],$$
    for some $u$ at some level $k\ge n$ by \cref{lemma: description of basal graph with rists}. Since
    $$[\mathrm{rist}_\rho(v)]=[\mathrm{rist}_G(u)^{\mathrm{N}_G(\mathrm{rist}_\rho(v))}]\ge [\mathrm{rist}_G(u)]$$
    we get 
    $$\mathrm{st}_{\rho}(v)=\mathrm{N}_G(\mathrm{rist}_\rho(v))\ge \mathrm{N}_G(\mathrm{rist}_G(u))=\mathrm{st}_G(u).$$
    Therefore
    \begin{align*}
        \mathrm{St}_{\rho}(n)&=\mathrm{Core}_{G}(\mathrm{st}_{\rho}(v))\ge \mathrm{Core}_{G}(\mathrm{st}_{G}(u))=\mathrm{St}_G(k).\qedhere
    \end{align*}
\end{proof}

\subsection{Previous results of Wilson}

Let $h\in G$. Following \cite{Wilson2}, we define the sets
$$X_h:=\{[h^{-1},h^k]\mid k\in G\}\text{, }\quad Y_h:=\{xyz\mid x,y,z\in X_h\}$$
and 
$$W_h:=\bigcup_{g\in G} \{Y_{h^g}\mid [Y_h,Y_{h^g}]\ne 1\}.$$
Note that
$$(X_h)^g=X_{h^g}\quad \text{and}\quad (Y_h)^g=Y_{h^g},$$
for any $g,h\in G$.

The following properties were proved by Wilson for rigid vertex stabilizers of branch groups in \cite{Wilson2}. However, his proofs only require weak branchness; see \cite[Lemmata 2.3, 3.2, 3.4 and the proof of Proposition 3.3]{Wilson2}. Note that \cite[Lemma~3.1]{Wilson2} may be replaced with Abért's general result for separating actions \cite[Theorem~1.1]{Abert}, which holds for all weakly branch groups. Furthermore, any basal subgroup may be realized as a rigid vertex stabilizer for some weakly branch action arguing as in \cite[Lemma 2.6(b)]{Wilson2}, so we may restate these results of Wilson for any basal subgroup of a weakly branch group:

\begin{proposition}
\label{proposition: properties of Wh}
    Let $G\le \mathrm{Aut}~T$ be a weakly branch group and $B\in B(G)$. Then
    \begin{enumerate}[\normalfont(i)]
        \item we have $\mathrm{C}_G^2(B)=B$;
        \item we have $\mathrm{C}_G^2(W_h)\le B$ for $h\in B$;
        \item there exists $1\ne h\in B$ such that $\mathrm{C}_G^2(W_h)=\bigcup_{g\in G} \{Y_{h^g}\mid g\in \mathrm{N}_G(B)\}$.
    \end{enumerate}
\end{proposition}

Following Wilson in \cite{Wilson2}, we show that for the values of $h\in B$ in \cref{proposition: properties of Wh}\textcolor{teal}{(iii)}, the subgroup $\mathrm{C}_G^2(W_h)$ is basal. Furthermore, any equivalence class on the structure graph $\mathcal{B}_G$ admits a basal representative of the form $\mathrm{C}_G^2(W_h)$:

\begin{lemma}
\label{lemma: definability of N(B)}
    Let $G\le \mathrm{Aut}~T$ be a weakly branch group and let $B\le G$ be a basal subgroup. Then there exists $1\ne h\in G$ such that $\mathrm{C}_G^2(W_h)$ is basal and $[B]=[\mathrm{C}_G^2(W_h)]$.
\end{lemma}
\begin{proof}
    Let $h$ be as in \cref{proposition: properties of Wh}\textcolor{teal}{(iii)}. Now consider $g\notin \mathrm{N}_G(B)$. Then by \cref{proposition: properties of Wh}\textcolor{teal}{(ii)} we get
    $$\mathrm{C}_G^2(W_h)^g\le B^g\ne B.$$
    Therefore $\mathrm{C}_G^2(W_h)^g\cap \mathrm{C}_G^2(W_h)=1$, so $\mathrm{N}_G(\mathrm{C}_G^2(W_h))\le \mathrm{N}_G(B)$. By \cref{proposition: properties of Wh}\textcolor{teal}{(iii)}, for any $g\in \mathrm{N}_G(B)$ we get
    $$\mathrm{C}_G^2(W_h)^g=\bigcup_{k\in G} \{Y_{h^k}\mid k\in \mathrm{N}_G(B)\}^g=\bigcup_{k\in G} \{Y_{h^k}\mid k\in \mathrm{N}_G(B)\}=\mathrm{C}_G^2(W_h).$$
    Therefore, as the subgroup $\mathrm{C}_G^2(W_h)$ is basal with normalizer
    \begin{align*}
        \mathrm{N}_G(\mathrm{C}_G^2(W_h))&=\mathrm{N}_G(B)
    \end{align*}
    and $\mathrm{C}_G^2(W_h)\le B$ by \cref{proposition: properties of Wh}\textcolor{teal}{(ii)}, we get $[B]=[\mathrm{C}_G^2(W_h)]$ as wanted.
\end{proof}

\subsection{First-order definability of the congruence topology and the level-stabilizers}

Let us fix a group $G$. A \textit{first-order sentence} in $G$ is a formula built up from the quantifiers $\lnot,\land, \lor,\rightarrow, \exists,\forall$ and equations $f(x_1,\dotsc,x_k)=1$, with variables taking values in $G$. A \textit{first-order definable} (subset) subgroup of $G$ is a (subset) subgroup $H\le G$ given by
$$H=\{g\in G\mid \varphi(g)\},$$
where $\varphi(g)$ is a first-order sentence in $G$. A collection of subgroups $\{H_i\}_{i\in I}$ is said to be \textit{first-order definable in $G$} if the set $I$ is first-order definable in $G$ and for every $i\in I$ the subgroup~$H_i$ is first-order definable in $G$.

Let us show first that normalizers and normal cores of first-order definable subgroups are themselves first-order definable:

\begin{lemma}
\label{lemma: first order logic of normalizers and normal cores}
    Let $G$ be a group and $H=\{g\in G\mid \varphi(g)\}\le G$ a first-order definable subgroup of $G$. Then the normalizer and the normal core of $H$ are both first-order definable in $G$. More precisely
    \begin{align*}
        \mathrm{N}_G(H)&=\{g\in G\mid (\forall y)(\varphi(y)\rightarrow \varphi(y^g))\},\\
        \mathrm{Core}_G(H)&=\{g\in G\mid (\forall y)(\varphi(g^y))\}.
    \end{align*}
\end{lemma}
\begin{proof}
    The condition $(\forall y)(\varphi(y)\rightarrow \varphi(y^g))$ says that for every element $y\in H$ we have $y^g\in H$, i.e. $H^g=H$. This is precisely the definition of $\mathrm{N}_G(H)$. For the normal core, note that an element $g\in \mathrm{Core}_G(H)$ if and only if $g\in H^y$ for every $y\in G$, i.e. if and only if $g^{y^{-1}}\in H$ for every $y\in G$. This is equivalent to the condition $(\forall y)(\varphi(g^y))$.
\end{proof}

\begin{proof}[Proof of \cref{Theorem: definability}\textcolor{teal}{$\mathrm{(i)}$} and \textcolor{teal}{$\mathrm{(iii)}$}]
By \textcolor{teal}{Lemmata} \ref{lemma: every basal contains rist} and \ref{lemma: equivalent topologies}, the congruence topology in $G$ is equivalent to the topology induced by the larger family of subgroups
$$\mathcal{N}:=\{\mathrm{Core}_G(\mathrm{N}_G(B))\mid B\in B(G)\}.$$
Then, to prove \cref{Theorem: definability}\textcolor{teal}{$\mathrm{(i)}$}, it is enough to show that the family $\mathcal{N}$ is first-order definable. In fact, \cref{Theorem: definability}\textcolor{teal}{$\mathrm{(iii)}$} will also follow from this once we establish \cref{proposition: characterization of rigidity}, as if $G$ is $T$-rigid, then by \cref{proposition: characterization of rigidity}, the family~$\mathcal{N}$ is precisely the family of level-stabilizers of $G$.

It was proved by Wilson in \cite[Theorem 4.1]{Wilson2} that $\mathrm{C}_G^2(W_h)$ is first-order definable for each $h\in G$, i.e.
    $$\mathrm{C}_G^2(W_h)=\{g\in G\mid \gamma_h(g)\}.$$
    Therefore $\mathrm{N}_G(\mathrm{C}_G^2(W_h))=\{g\in G\mid \varphi_h(g)\}$ for 
    \begin{align*}
        \varphi_h(g)&:~(\forall y)(\gamma_h(y)\rightarrow \gamma_h(y^g)),
    \end{align*}
and  $\mathrm{Core}_G(\mathrm{N}_G(\mathrm{C}_G^2(W_h)))=\{g\in G\mid \sigma_h(g)\}$ for
\begin{align*}
    \sigma_h(g)&:~(\forall y)(\varphi_h(g^y)),
\end{align*}
by \cref{lemma: first order logic of normalizers and normal cores}. Wilson further defined in \cite[Theorem 4.1]{Wilson2} the first-order sentence~$\beta(g)$:
$$g\ne 1\land (\forall y)\big(((\exists z_1\exists z_2)\gamma_g(z_1)\land \gamma_{g^y}(z_2)\land [z_1,z_2]\ne 1 )\rightarrow ((\forall z_3)(\gamma_g(z_3)\rightarrow \gamma_{g^y}(z_3)))\big),$$
and showed that $\beta(h)$ holds if and only if $h\ne 1$ and the subgroup $\mathrm{C}_G^2(W_h)$ is basal. By \cref{lemma: definability of N(B)}, for any $B\in B(G)$, there exists some $1\ne h\in G$ such that $\beta(h)$ is true and $\mathrm{N}_G(B)=\mathrm{N}_G(\mathrm{C}_G^2(W_h))$. Therefore
$$\mathcal{N}=\{\mathrm{Core}_G(\mathrm{N}_G(\mathrm{C}_G^2(W_h)))\mid h\in S\},$$
where $S:=\{h\in G\mid \beta(h)\}$ and $$\mathrm{Core}_G(\mathrm{N}_G(\mathrm{C}_G^2(W_h)))=\{g\in G\mid \sigma_h(g)\}$$
for each $h\in S$, proving $\mathcal{N}$ is first-order definable.
\end{proof}

\subsection{Definability of the structure graph}

The definable set $S:=\{h\in G\mid \beta(h)\}$ appearing in the proof above is a union of conjugacy classes, and we may define an action of $G$ on $S$ by conjugation. We also define $\delta(g,h)$ via
$$\delta(g,h):~ (\exists z)(\gamma_g(z)\wedge \gamma_h(z)\wedge z\ne 1)\wedge(\forall y)(\varphi_g(y)\rightarrow \varphi_h(y)).$$
For $g,h\in S$ the statement $\delta(g,h)$ holds if and only if $\mathrm{C}_G^2(W_g)\cap \mathrm{C}_G^2(W_h)\ne 1$ and 
$$\mathrm{N}_G(\mathrm{C}_G^2(W_g))\le \mathrm{N}_G(\mathrm{C}_G^2(W_h)),$$
which yields a partial order in $S$. Therefore $\delta(g,h)\land\delta(h,g)$ yields an equivalence relation $\sim$ in $S$ and $G$ acts on $S/\sim$ by conjugation. By construction, the first-order definable quotient $S/\sim$ is $G$-equivariantly isomorphic (as a partially ordered set) to the structure graph $\mathcal{B}_G$. This proves \cref{Theorem: definability}\textcolor{teal}{(ii)}, and generalizes \cite[Theorem~4.1(d)]{Wilson2} to weakly branch groups.

We remark that the statement $\delta(g,h)$ is a slight variation of the analogous statement in \cite[Theorem~4.1]{Wilson2}. This is due to replacing $\sim_S$ with $\sim$ in the definition of the structure graph.

\section{Rigidity of weakly branch actions: the general case}
\label{section: rigidity}

The goal of this section is to characterize, in terms of the structure graph $\mathcal{B}_G$, when two weakly branch actions on the same spherically homogeneous rooted tree $T$ are conjugate in $\mathrm{Aut}~T$. We also study the connection between the rigidity of a weakly branch group and the rigidity of its congruence completion. We fix the spherically homogeneous rooted tree $T$ for the remainder of the section.

\subsection{$T$-filtrations and weakly branch actions}

A group $G$ is defined to be $T$-\textit{rigid} if it admits a weakly branch action on $T$ and for any pair of weakly branch actions $\rho,\tau:G\to \mathrm{Aut}~T$ the respective images $\rho(G)$ and $\tau(G)$ are conjugate in $\mathrm{Aut}~T$.

Grigorchuk and Wilson introduced in \cite{GrigorchukWilson} two properties, which guarantee rigidity of branch actions of $G$ on a spherically homogeneous rooted tree $T$:
\begin{enumerate}
    \item[$(*)$] for each vertex $v\in T$ the stabilizer $\mathrm{st}_G(v)$ acts as a (transitive) cyclic group of prime order on the immediate descendants of $v$;
    \item[$(**)$] whenever $u$ and $u'$ are incomparable vertices and $v$ is a descendant of (not equal to) $u$, there is an element $g\in G$ fixing $u'$ and moving $v$.
\end{enumerate}

Note that property $(*)$ forces $T$ to be \textit{locally prime}, i.e. every vertex of $T$ has a prime number of immediate descendants, where the primes corresponding to vertices at different levels might be distinct. Garrido slightly generalized property~$(*)$ by asking only for a primitive action of the stabilizer and property $(**)$ by asking the vertex $u$ to also be fixed by $g$; see properties (A) and (B) in \cite[Proposition~7.1.4]{GarridoPhD}. Note however, that there is an issue in the proof of \cite[Proposition~7.1.4]{GarridoPhD} (vertex stabilizers of incomparable vertices can be contained one in the other, indeed they can even be equal), so properties (A) and (B) are sufficient but not necessary conditions for $\mathcal{B}_G$ to be isomorphic to $T$.

Properties $(*)$ and $(**)$ were further extended by Hardy in \cite{Hardy}, where he introduced the following property for a branch action $\rho:G\to \mathrm{Aut}~T$:
\begin{enumerate}
    \item[$(\mathrm{S})$] for each vertex $v\in T$ and each subgroup $L\le G$, the containment $L\ge \mathrm{st}_\rho(v)$ implies $L=\mathrm{st}_\rho(w)$ for some $w$ above $v$.
\end{enumerate}

We shall restate Hardy's property (S) as follows. Let us define the set $\mathcal{L}_\rho$ via
$$\mathcal{L}_{\rho}:=\{L\le G~\mid~L\ge \mathrm{st}_\rho(u)\text{ for some }u\in T\}.$$
We consider the subset $\mathcal{S}_{\rho}\subseteq \mathcal{L}_\rho$ of subgroups consisting of vertex stabilizers (for the branch action $\rho$). Hardy's property (S) is simply the equality $$\mathcal{L}_\rho=\mathcal{S}_{\rho}.$$
In \cite[Theorem 15.4.2]{Hardy}, Hardy characterized branch groups whose structure graph is a tree precisely as those admitting a branch action satisfying property (S).

In this paper, we consider weakly branch actions of a group $G$. Note that both properties $(*)$ and $(**)$ and property (S) may also be considered for weakly branch actions. In fact, by \cref{lemma: actions as embeddings}, property (S) (for weakly branch groups) makes the $G$-equivariant map $\phi_\rho$ surjective, and thus bijective, which yields $T$-rigidity: 

\begin{proposition}
\label{proposition: rigidity Hardy}
    Let $\rho:G\to \mathrm{Aut}~T$ be a weakly branch action. Then $\rho$ has property $\mathrm{(S)}$ if and only if $\phi_\rho:T\to \mathcal{B}_G$ is an isomorphism.
\end{proposition}

However, a (weakly) branch group may be $T$-rigid even if $\phi_\rho:T\to \mathcal{B}_G$ is not an isomorphism, as we shall see. Therefore, a more subtle property is needed to characterize $T$-rigidity. 

First, note that the definition of the set $\mathcal{L}_\rho$ is actually independent of the weakly branch action $\rho:G\to\mathrm{Aut}~T$. Indeed, by \cref{lemma: every basal contains rist}, if $\tau:G\to\mathrm{Aut}~T$ is another weakly branch action, for any $v\in T$, there exists some $u\in T$ such that
$$[\mathrm{rist}_\tau(v)]\ge[\mathrm{rist}_\rho(u)],$$
and thus
$$\mathrm{st}_\tau(v)=\mathrm{N}_G(\mathrm{rist}_\tau(v))\ge \mathrm{N}_G(\mathrm{rist}_\rho(u))=\mathrm{st}_\rho(u).$$
Therefore, we shall write $\mathcal{L}_G$ in the following to refer to the set $\mathcal{L}_\rho$. In fact, by \cref{lemma: properties of basal subgroups}\textcolor{teal}{(ii)} and \cref{lemma: every basal contains rist}, subgroups in $\mathcal{L}_G$ are precisely the normalizers of basal subgroups of $G$:
$$\mathcal{L}_G=\{\mathrm{N}_G(B)\mid B\in B(G)\}.$$

Now, we introduce the notion of a $T$-filtration:

\begin{definition}[$T$-filtration]
\label{definition: T filtration}
    A collection of subgroups $\{H_k\}_{k\ge 0}$ of $G$  is defined to be a \textit{$T$-filtration} if $H_0=G$ and, for every $k\ge 0$:
\begin{enumerate}[\normalfont(i)]
    \item $H_k\ge H_{k+1}$;
    \item $H_k\in \mathcal{L}_G$;
    \item $|G:H_{k}|=N_k$, where $N_k$ is the number of vertices at the $k$th level of $T$.
\end{enumerate}
\end{definition}

Note that we shall be writing $N_k$ for the number of vertices at the $k$th level of $T$ for the remainder of the paper. $T$-filtrations may be used to define weakly branch actions of $G$ on $T$:

\begin{lemma}
\label{lemma: T filtration define branch actions}
    Let $\{H_k\}_{k\ge 0}$ be a $T$-filtration. Then for every $v\in T$, if $k$ is the level at which $v$ lies, there exist $u_k\in T$ and $g_v\in G$ such that the map $\phi_\tau:T\to \mathcal{B}_G$ 
    $$v\mapsto [\mathrm{rist}_\rho(u_k)^{{H_k}g_v}]=[\mathrm{rist}_\rho(u_k^{\rho(g_v)})^{{H_k}^{g_v}}]$$
    is injective and $G$-equivariant.
\end{lemma}
\begin{proof}
     By definition, for each $k\ge 0$ we have $H_k\in \mathcal{L}_G$. Thus, for every $k\ge 0$, there exists some $u_k\in T$ such that $H_k\ge \mathrm{st}_\rho(u_k)$. The condition on the index of $H_k$ in~$G$ implies that $u_k$ lies at level $k$ or below. Let $w_k$ be the unique vertex at level~$k$ above~$u_k$. Then the assignment
     $$w_k\mapsto [\mathrm{rist}_\rho(u_k)^{H_k}]$$
     is well defined as the image of $w_k$ does not depend on the choice of $u_k$ (any vertex~$u$ such that $H_k\ge \mathrm{st}_\rho(u)$ yields the same equivalence class in $\mathcal{B}_G$ by \cref{lemma: properties of basal subgroups}\textcolor{teal}{(ii)}). Furthermore, this assignment can be extended to a $G$-equivariant map $\phi_\tau:T\to \mathcal{B}_G$. Indeed, if $v\in T$ is any vertex and $k$ is the level of $T$ at which $v$ lies, by the level-transitivity of $\rho(G)$, there exists some $g_v\in G$ such that $w_k^{\rho(g_v)}=v$. Then 
\begin{align*}
    v&=w_k^{\rho(g_v)}\mapsto [\mathrm{rist}_\rho(u_k)^{H_k}]^{g_v}=[\mathrm{rist}_\rho(u_k)^{H_k g_v}]=[\mathrm{rist}_\rho(u_k^{\rho(g_v)})^{{H_k}^{g_v}}]
\end{align*}
is a well-defined $G$-equivariant map $\phi_\tau:T\to \mathcal{B}_G$. If $u$ and $v$ are two vertices at levels $k\ne \ell$ respectively, then $\phi_\tau(u)\ne \phi_\tau(v)$, as the stabilizers of $\phi_\tau(u)$ and $\phi_\tau(v)$, namely $H_k^{g_u}$ and $H_\ell^{g_v}$, have distinct indices in $G$. Finally, if $u$ and $v$ both lie at the same level~$k$, then $\phi_\tau(u)\ne \phi_\tau(v)$ too. Indeed $H_k$ has exactly $N_k$ distinct right cosets by definition. By the level-transitivity of $\rho(G)$, all of them are realized as the image of some vertex at the $k$th level of $T$, and there are exactly $N_k$ vertices at the $k$th level of $T$. Thus $\phi_\tau$ must be injective.
\end{proof}

\subsection{Proofs of the main results}

Let us denote by $\mathcal{F}_T$ the subset of $\mathcal{L}_G$ consisting of those subgroups in $\mathcal{L}_G$ belonging to some $T$-filtration, i.e.
$$\mathcal{F}_T:=\{H\in \mathcal{L}_G\mid \text{there is a }T\text{-filtration }\{H_k\}_{k\ge 0} \text{ and }k\ge 0,\text{ where }H_k=H\}.$$
In particular, we have $\mathcal{F}_T\supseteq \mathcal{S}_{\rho}$ for any weakly branch action $\rho:G\to \mathrm{Aut}~T$. We are now ready to prove \textcolor{teal}{Theorems} \ref{proposition: characterization of rigidity} and \ref{Theorem: T-rigidity in finite quotients}:

\begin{proof}[Proof of \cref{proposition: characterization of rigidity}]
    Note that (ii)$\implies$(iii) is obvious. We prove first (iii)$\implies$(i). If $\mathcal{F}_T=\mathcal{S}_{\rho}$ holds for some weakly branch action $\rho:G\to\mathrm{Aut}~T$, we show that for any other weakly branch action $\tau:G\to \mathrm{Aut}~T$, for each $v\in T$ we must have $$[\mathrm{rist}_\tau(v)]=[\mathrm{rist}_\rho(w)]$$
    for some $w$ at the same level of $T$ as $v$. This yields that there is a unique weakly branch action of $G$ up to conjugation by an automorphism of $T$, and thus $G$ is $T$-rigid. By \cref{lemma: description of basal graph with rists}, we have
    $$[\mathrm{rist}_\tau(v)]=[\mathrm{rist}_\rho(u)^{\mathrm{N}_G(\mathrm{rist}_\tau(v))}]$$
    for some $u\in T$. Therefore, we have 
    $$\mathrm{N}_G(\mathrm{rist}_\tau(v))=\mathrm{st}_\tau(v)\in \mathcal{S}_{\tau}\subseteq \mathcal{F}_T=\mathcal{S}_{\rho}.$$
    Then $\mathrm{st}_\tau(v)=\mathrm{st}_\rho(w)$ for the unique $w$ at the same level of $T$ as $v$ and above $u$. Thus
    $$[\mathrm{rist}_\tau(v)]=[\mathrm{rist}_\rho(u)^{\mathrm{N}_G(\mathrm{rist}_\tau(v))}]=[\mathrm{rist}_\rho(u)^{\mathrm{st}_\tau(v)}]=[\mathrm{rist}_\rho(u)^{\mathrm{st}_\rho(w)}]=[\mathrm{rist}_\rho(w)].$$
    
    For (i)$\implies$(ii), let us assume by contradiction that there exists a weakly branch action $\rho:G\to \mathrm{Aut}~T$ such that  $\mathcal{S}_{\rho}\subsetneq \mathcal{F}_T$. Then, we may define a second weakly branch action $\tau:G\to \mathrm{Aut}~T$, where there is a vertex $v\in T$ such that $\mathrm{st}_\tau(v)\notin \mathcal{S}_\rho$. The existence of such a weakly branch action is guaranteed by \cref{lemma: T filtration define branch actions}. Since~$G$ is $T$-rigid, there exists $f\in \mathrm{Aut}~T$ such that $\rho(G)=\tau(G)^f$. Then
    $$(v^{\tau(g)})^f=(v^f)^{\rho(g)}$$
    for every $g\in G$. Now, there exists $g\in \mathrm{st}_\rho(v^f)$ such that $v^{\tau(g)}\ne v$. Indeed, otherwise 
    $$\mathrm{st}_\tau(v)\ge \mathrm{st}_\rho(v^f).$$
    Since, by level-transitivity, the index of a vertex stabilizer is the number of vertices at the level at which the vertex lies, we would further get the equality 
    $$\mathrm{st}_\tau(v)=\mathrm{st}_\rho(v^f)\in \mathcal{S}_\rho$$
    as $f\in \mathrm{Aut}~T$. However, this contradicts our assumption on $\tau$. Therefore, for such an element $g$, we have  $v^{\tau(g)}\ne v$, and we obtain
    $$(v^{\tau(g)})^f=(v^f)^{\rho(g)}=v^f.$$
    Hence $f\notin \mathrm{Aut}~T$, which yields a contradiction.    
\end{proof}

\begin{proof}[Proof of \cref{Theorem: T-rigidity in finite quotients}]

    Let us consider a weakly branch action $\rho:G\to \mathrm{Aut}~T$. Then, the closure $\overline{\rho(G)}\le \mathrm{Aut}~T$ is also a weakly branch group as $$\mathrm{rist}_H(v)=H\cap\mathrm{rist}(v)\le \overline{H}\cap \mathrm{rist}(v)=\mathrm{rist}_{\overline{H}}(v)$$
    for any $H\le \mathrm{Aut}~T$. Furthermore, by \cref{Theorem: definability}, the closure $\overline{\rho(G)}$ is isomorphic to the completion $\overline{G}$ of $G$ with respect to the congruence topology, which induces a weakly branch action $\overline{\rho}:\overline{G}\to \mathrm{Aut}~T$ via $\overline{G}\mapsto \overline{\rho(G)}$. We show that $\mathcal{F}_T=\mathcal{S}_\rho$ for~$G$ if and only if $\mathcal{F}_T=\mathcal{S}_{\overline{\rho}}$ for $\overline{G}$. Then the result follows from \cref{proposition: characterization of rigidity}.
    
    Let us assume that $G=\rho(G),\overline{G}=\overline{\rho}(\overline{G})\le \mathrm{Aut}~T$ for the remainder of the proof to simplify notation. First, recall the standard property of profinite topologies that the closure map $H\mapsto \overline{H}$ induces an index-preserving one-to-one correspondence between the open subgroups of~$G$ and those of $\overline{G}$. Now, every subgroup in $\mathcal{F}_T$ is open in the congruence topology (as each subgroup in $\mathcal{F}_T$ contains a vertex stabilizer and thus, a level stabilizer), so the result will follow once we establish the equality $\overline{\mathrm{st}_G(v)}=\mathrm{st}_{\overline{G}}(v)$ for each $v\in T$. Indeed, then the closure map $H\mapsto \overline{H}$ induces a bijection between the sets~$\mathcal{F}_T$ corresponding to $G$ and $\overline{G}$ as the closure map is index-preserving, which restricts to a bijection of the corresponding sets $\mathcal{S}_\rho$ and $\mathcal{S}_{\overline{\rho}}$. Hence $\mathcal{F}_T=\mathcal{S}_\rho$ for $G$ if and only if $\mathcal{F}_T=\mathcal{S}_{\overline{\rho}}$ for~$\overline{G}$.
    
    As $\mathrm{st}_G(v)=G\cap \mathrm{st}(v)$ and $\mathrm{Aut}~T$ is residually finite, we have the inclusion
    $$\mathrm{st}_G(v)=G\cap \mathrm{st}(v)\le \overline{G}\cap \mathrm{st}(v)=\mathrm{st}_{\overline{G}}(v),$$
    and thus
    $$\overline{\mathrm{st}_G(v)}\le \mathrm{st}_{\overline{G}}(v)$$
    as $\mathrm{st}_{\overline{G}}(v)$ is closed. The equality of subgroups follows from the equality of indices $$|\overline{G}:\overline{\mathrm{st}_G(v)}|=|G:\mathrm{st}_G(v)|=|\overline{G}:\mathrm{st}_{\overline{G}}(v)|,$$ 
    where the first equality follows from the closure map $H\mapsto \overline{H}$ preserving subgroup indices and the second equality follows from the orbit-stabilizer theorem, as both~$G$ and $\overline{G}$ are level-transitive in $\mathrm{Aut}~T$.
\end{proof}

\section{Rigidity of weakly branch actions: fractal groups of $p$-adic automorphisms}
\label{section: p-adic}

In this section, we generalize property ($**$) of Grigorchuk and Wilson and relate the Hausdorff dimension of a fractal weakly branch group in $W_p$ to the rigidity of its weakly branch actions. We give a characterization of $T_p$-rigidity of fractal weakly branch groups of $p$-adic automorphisms, which is very easy to check. We conclude the section with examples of further applications.

\subsection{A sufficient condition for rigidity}

We first state the following straightforward but very useful lemma, which can be proved as in \cite[Proposition 4.2]{RestrictedSpectra}:

\begin{lemma}
    \label{lemma: level-transitivity}
    A group $G\le \mathrm{Aut}~T$ acts level-transitively on $T$ if and only if for every $v\in T$ the subgroup $\mathrm{st}_G(v)$ acts level-transitively on the subtree $T_v$ rooted at~$v$.
\end{lemma}

We need the following key lemma, in the spirit of \cite[Lemma 5.1(a)]{WilsonBook}:

\begin{lemma}
\label{lemma: orbits of elements determine basal}
    Let $\rho:G\to \mathrm{Aut}~T$ be a weakly branch action and let $H_1,H_2$ be subgroups of $G$ such that $H_1,H_2\ge \mathrm{st}_\rho(u)$ for some $u\in T$. Then, we have $H_1=H_2$ if and only if the $\rho(H_1)$-orbit of $u$ coincides with the $\rho(H_2)$-orbit of $u$.
\end{lemma}
\begin{proof}
    It follows from \cref{lemma: properties of basal subgroups}. Indeed, by \cref{lemma: properties of basal subgroups}, both
    $$B_1:=\mathrm{rist}_\rho(u)^{H_1}\quad \text{and}\quad B_2:=\mathrm{rist}_\rho(u)^{H_2}$$
    are basal subgroups with normalizer equal to $H_1$ and $H_2$ respectively. Then, if the $\rho(H_1)$-orbit of $u$ coincides with the $\rho(H_2)$-orbit of $u$, we get $B_1=B_2$, and hence
    $$H_1=\mathrm{N}_G(B_1)=\mathrm{N}_G(B_2)=H_2.$$
    The converse is trivial.
\end{proof}

In the following proposition, for a vertex $v_k\in T$, we write $p_k$ for the (prime) number of immediate descendants of $v_k$ in $T$.

\begin{proposition}
\label{theorem: sufficient condition for rigidity in p-adic}
    Let $G\le \mathrm{Aut}~T$ be a weakly branch group satisfying $(*)$ and the following property:
    \begin{enumerate}
        \item[$(\mathrm{N})$] For every triple of vertices $v_1,v_2,v_3\in T$, where $v_{k+1}$ is an immediate descendant of $v_k$ for $1\le k\le 2$, we have either $p_1\ne p_2$ or
    $$|\mathrm{st}_G(v_2):\mathrm{Core}_{\mathrm{st}_G(v_1)}(\mathrm{st}_G(v_3))|>p_1=p_2,$$
    with $p_1$ and $p_2$ being the (prime) number of immediate descendants of $v_1$ and $v_2$ respectively.
    \end{enumerate}
     Then $G$ has Hardy's property $\mathrm{(S)}$ and, in particular $G$ is $T$-rigid.
\end{proposition}
\begin{proof}
    First note that it is enough to prove that if $\mathrm{st}_G(w_n)\le H< \mathrm{st}_G(v_1)$ and $H$ is of index $p_1$ in $\mathrm{st}_G(v_1)$, then $H$ is a vertex stabilizer. Then the result follows by induction on the index of $H\in \mathcal{F}_T$ in $G$. 
    
    Assume first that $H$ fixes all the immediate descendants of $v_1$. Then $H=\mathrm{st}_G(v_2)$ for an immediate descendant $v_2$ of $v_1$ as $H$ has index $p_1$ in $\mathrm{st}_G(v_1)$. Let us assume now that $H$ moves some immediate descendant $v_2$ of $v_1$. Since $G$ has property ($*$), this implies that $H$ acts transitively on the immediate descendants of $v_1$. Let us consider $\mathrm{st}_H(v_2)=H\cap \mathrm{st}_G(v_2)$. Then $\mathrm{st}_H(v_2)$ has index dividing $p_1$ in $\mathrm{st}_G(v_2)$, i.e. it is either 1 or $p_1$. If the index is 1 then $H\ge \mathrm{st}_G(v_2)$ and $H$ acts level-transitively on the subtree rooted at~$v_1$ by \cref{lemma: level-transitivity}; hence $H=\mathrm{st}_G(v_1)$ by \cref{lemma: orbits of elements determine basal}, which is a contradiction. Thus, let us assume the index is $p_1$. Then, as by assumption either $p_1\ne p_2$ or
    $$|\mathrm{st}_G(v_2):\mathrm{Core}_{\mathrm{st}_G(v_1)}(\mathrm{st}_G(v_3))|>p_1=p_2,$$
    the subgroup $\mathrm{st}_H(v_2)$ acts non-trivially on the immediate descendants of $v_2$, and thus, transitively as $G$ has property ($*$). Iterating this argument shows that $H$ must act level-transitively on the subtree rooted at $v_1$. Hence, by \cref{lemma: orbits of elements determine basal}, we obtain the equality $H=\mathrm{st}_G(v_1)$, which yields again a contradiction.
\end{proof}

An immediate corollary to \cref{theorem: sufficient condition for rigidity in p-adic} is that if the primes $\{p_k\}_{k\ge 0}$ are all distinct, we obtain rigidity:

\begin{corollary}
    Let $\{p_k\}_{k\ge 0}$ be a collection of distinct primes. Let $T$ be a spherically homogeneous rooted tree such that for each $k\ge 0$, the vertices at level $k$ have~$p_k$ immediate descendants. Then, any weakly branch group $G\le \mathrm{Aut}~T$ satisfying $(*)$ is $T$-rigid.
\end{corollary}

\begin{remark}
\label{remark: normality}
    The following was pointed out to the author by Fernández-Alcober. The condition 
    $$|\mathrm{st}_G(v_2):\mathrm{Core}_{\mathrm{st}_G(v_1)}(\mathrm{st}_G(v_3))|>p_2,$$
    is in fact equivalent to $\mathrm{st}_G(v_3)$ not being normal in $\mathrm{st}_G(v_1)$. Indeed, the subgroup $\mathrm{st}_G(v_3)$ is normal in $\mathrm{st}_G(v_1)$ if and only if it is equal to its normal core in $\mathrm{st}_G(v_1)$. Then, the equivalence follows from $$|\mathrm{st}_G(v_2):\mathrm{st}_G(v_3)|=p_2.$$
    This motivates the name (N) for the condition in \cref{theorem: sufficient condition for rigidity in p-adic}.
\end{remark}

    Property (N) generalizes property $(**)$ of Grigorchuk and Wilson. Indeed, let us consider any $v_1\in T$, any two distinct immediate descendants of $v_1$, namely $v_2$ and~$v_2'$, and two immediate descendants $v_3$ and $v_3'$ of $v_2$ and $v_2'$ respectively. Then, by property $(**)$, there exists $g\in \mathrm{st}_G(v_3)$ such that $(v_3')^g\ne v_3'$. Thus
$$|\mathrm{st}_G(v_2):\mathrm{Core}_{\mathrm{st}_G(v_1)}(\mathrm{st}_G(v_3))|\ge p_2^2,$$
and $G$ satisfies the second assumption in property (N). 

In the remainder of the section, we specialize \cref{theorem: sufficient condition for rigidity in p-adic} to fractal groups of $p$-adic automorphisms, which leads to the characterization in \cref{proposition: characterization of rigidity}; but first, let us introduce fractal groups and certain related notions.

\subsection{Self-similarity and fractality}
Let $\sigma:=(1\, \dotsb \, p)\in \mathrm{Sym}(p)$ be a $p$-cycle for some prime~$p$. Recall that we identify a regular rooted tree with the corresponding free monoid. We define the group of \textit{$p$-adic automorphisms} $W_p$ as
$$W_p:=\{g\in \mathrm{Aut}~T_p\mid g|_v^1\in \langle \sigma\rangle \text{ for all }v\in T_p\},$$
where $g|_v^1$ is the \textit{label of $g$ at $v$}, i.e. the unique permutation $g|_v^1\in \mathrm{Sym}(p)$ such that
$$(vx)^g=v^gx^{g|_v^1}$$
for every $1\le x\le p$. For $p=2$, we have the equality $W_2=\mathrm{Aut}~T_2$.

More generally, for any $g\in \mathrm{Aut}~T$, any $v\in T$ and any $1\le n\le \infty$, we define~$g|_v^n$ the \textit{section of $g$ at $v$ of depth $n$}, as the unique automorphism $g|_v^n\in \mathrm{Aut}~T_v^{[n]}$ of the \textit{$n$th truncated tree $T_v^{[n]}$ at }$v$, i.e. the finite tree obtained by truncating the subtree~$T_v$ at level $n$, such that
$$(vu)^g=v^gu^{g|_v^n}$$
for every $u$ at level $k\le n$. For $n=\infty$, we have $T_v^{[\infty]}=T_v$ and we shall drop the superscript and write $g|_v$ as in the introduction. We have an isomorphism $$\psi_n:\mathrm{Aut}~T\to \big(\mathrm{Aut}~T_{v_1}\times\ldots\times \mathrm{Aut}~T_{v_{N_n}}\big)\rtimes \mathrm{Aut}~T^{[n]}$$
defined via
$$g\mapsto(g|_{v_1},\dotsc, g|_{v_{N_n}})g|_\emptyset^n,$$
where $\emptyset$ is the root of $T$ and $N_n$ is the number of vertices at level $n$. We shall write~$\psi$ for $\psi_1$.

Let $G\le \mathrm{Aut}~T$. We also define the homomorphism $\varphi_v:\mathrm{st}_G(v)\to \mathrm{Aut}~T_v$ via $g\mapsto g|_v$. For the $d$-regular rooted tree $T_d$, the subtrees rooted at any vertex are isomorphic to $T_d$. Thus, for any $g\in \mathrm{Aut}~T_d$ and any $v\in T_d$, we may regard $g|_v$ as an element of $\mathrm{Aut}~T_d$ under this identification. We say that $G\le \mathrm{Aut}~T_d$ is \textit{self-similar} if $g|_v\in G$ for every $g\in G$ and $v\in T_d$, and \textit{fractal} if $G$ is self-similar, level-transitive and $\varphi_v$ is surjective for every $v\in T_d$.

Recall that a closed subgroup $F\le \mathrm{Aut}~T_d$ is said to be of \textit{finite type} if there exists some $D\ge 1$ called the \textit{depth} and a set of \textit{patterns} $H\le \mathrm{Aut}~T_d^{[D]}$ such that
$$F=\{g\in \mathrm{Aut}~T_d\mid g|_v^D\in H \text{ for all }v\in T_d\}.$$

Note that a self-similar group $G\le \mathrm{Aut}~T_d$ embeds, for every $n\ge 1$, in the group of finite type given by the set of patterns $G/\mathrm{St}_G(n)$; see \cite[Section 3.5]{RestrictedSpectra}.

\subsection{Rigidity in the $p$-adic tree}

\cref{theorem: sufficient condition for rigidity in p-adic} takes an especially simple form for fractal groups, as one only needs to check condition (N) at the top of the tree. Recall that if $G$ is fractal and fulfills condition $(*)$, then $G\le W_p$ for some prime $p\ge 2$. We are now ready to prove \cref{Corollary: sufficient for p-adic}:

\begin{proof}[Proof of \cref{Corollary: sufficient for p-adic}]
    Let us prove first that the condition (N) is fulfilled if and only if $|G:\mathrm{St}_G(2)|>p^2$.

    Since $G$ is fractal, for any vertex $v\in T_p$, we get the equality $\varphi_v(\mathrm{st}_G(v))=G$. Thus, condition (N) is satisfied if and only if it is satisfied with $v_1=\emptyset$. Since $G\le W_p$, we have $\mathrm{st}_G(v)=\mathrm{St}_G(1)$ for any $v$ at the first level and
    $$\mathrm{Core}_{G}(\mathrm{st}_G(w))=\mathrm{St}_G(2)$$
    for any $w$ at the second level. Thus, since $|G:\mathrm{St}_G(1)|=p$, we get 
    \begin{align*}
        |\mathrm{St}_G(1):\mathrm{St}_G(2)|&>p \quad \text{if and only if}\quad |G:\mathrm{St}_G(2)|>p^2.
    \end{align*}
    
    Let us assume now that $|G:\mathrm{St}_G(2)|=p^2$. Then, either 
    $$G/\mathrm{St}_G(2)\cong C_p\times C_p\quad \text{or}\quad G/\mathrm{St}_G(2)\cong C_{p^2}.$$
    We prove $T_p$-rigidity  if $G/\mathrm{St}_G(2)\cong C_{p^2}$. Under this assumption, the group $G$ embeds in the group of finite type $F$ given by the set of patterns:
    $$G/\mathrm{St}_G(2)=\langle \xi\rangle,$$
    where $\xi$ is a transitive cycle in $\mathrm{Aut}~T^{[2]}_p\cong C_p\wr C_p$. Note that an element $\zeta \in \langle \xi\rangle$ acts non-trivially on the first level of~$T_p^{[2]}$ if and only if $\zeta$ has order $p^2$.

    Let us fix $g\in F$ and $v\in T_p$ a vertex fixed by $g$. Then, either $g|_v$ fixes the immediate descendants of $v$ or $g|_v^2$ has order $p^2$. If $g|_v^2$ has order $p^2$, we get
    $$g^p|_v^2=(g|_v^2)^p\ne 1,$$
    as $v$ is fixed by $g$. Therefore $g^p|_{w}^2$ has order $p^2$ for any immediate descendant $w$ of $v$, as $g^p|_{w}^2$ acts non-trivially on the immediate descendants of $w$.  We get the level-transitivity of $g|_v$ iterating this argument. Hence, we obtain the following dichotomy:
    \begin{enumerate}[\normalfont(i)]
        \item either $g$ fixes the immediate descendants of $v$; or
        \item $g|_v$ acts level-transitively on $T_p$.
    \end{enumerate}
    Let $H\le G$ be such that:
    \begin{enumerate}[\normalfont(a)]
        \item $\mathrm{st}_G(w_n)\le H\le \mathrm{st}_G(v_1)$; and
        \item $H$ has index $p$ in $\mathrm{st}_G(v_1)$.
    \end{enumerate}
    Then, using the above dichotomy (i)-(ii) and \cref{lemma: orbits of elements determine basal}, we get $H=\mathrm{st}_G(v_2)$, where~$v_2$ is an immediate descendant of~$v_1$. Therefore~$G$ is $T_p$-rigid, arguing by induction on the index of $H$ in $G$ and applying \cref{proposition: characterization of rigidity}.

    To conclude, let us prove the converse, i.e. $G/\mathrm{St}_G(2)\cong C_p\times C_p$ implies that $G$ is not $T_p$-rigid. Note that by \cref{remark: normality}, we have $\mathrm{St}_G(2)=\mathrm{st}_G(v)$ for any vertex~$v$ at level 2. Let us write $G/\mathrm{St}_G(2)=H_1\times H_2$, where $H_1$ and $H_2$ correspond to $G/\mathrm{St}_G(1)$ and $\mathrm{St}_G(1)/\mathrm{St}_G(2)$ respectively. Then, by the third isomorphism theorem, we may define the unique subgroup $K\le G$ such that $\mathrm{St}_G(2)\le K$ and $K/\mathrm{St}_G(2)=H_1$. Then, the subgroup~$K$ contains $\mathrm{St}_G(2)=\mathrm{st}_G(v)$ and it has index $p$, so $K\in \mathcal{F}_{T_p}$. However $K$ moves every vertex in $T_p$ but the root, so  $K$ is not a vertex stabilizer. Therefore $G$ is not $T_p$-rigid by \cref{proposition: characterization of rigidity}.
\end{proof}

\subsection{Hausdorff dimension and rigidity}

Hausdorff dimension is strongly related to weak branchness. The author has proved recently that a weakly regular branch group has positive Hausdorff dimension \cite[Proposition~5.3]{RestrictedSpectra} and that a self-similar group with positive Hausdorff dimension in $W_p$ is weakly branch \cite[Theorem~C]{AV}. More generally, almost all of the projections of a positive-dimensional subgroup of~$W_p$ must be weakly branch \cite[Theorem 2.1]{ArborealJorge}.

In order to prove \cref{Corollary: hdim and rigidity}, we shall need the tools to compute the Hausdorff dimension developed by the author in \cite{RestrictedSpectra}. The sequence $\{s_n(G)\}_{n\ge 1}$, given by
$$s_n(G):=p\log_p|\mathrm{St}_G(n-1):\mathrm{St}_G(n)|-\log_p|\mathrm{St}_G(n):\mathrm{St}_G(n+1)|$$
was first introduced in \cite{GeneralizedBasilica}. In \cite{RestrictedSpectra}, the author considered the ordinary generating function 
$$ S_{G}(x):=\sum_{n\ge 1}s_n(G)x^n$$
and proved the following: 

\begin{theorem}[{see {\cite[Theorem B and Proposition 1.1]{RestrictedSpectra}}}]
\label{theorem: formula hdim jorge}
Assume that $G\le W_p$ is self-similar. Then $\{s_n(G)\}_{n\ge 1}$ is non-negative, the Hausdorff dimension of the closure of $G$ in $W_p$ is given by a proper limit and it may be computed as
$$\mathrm{hdim}_{W_p}(\overline{G})=1-S_{G}(1/p).$$
\end{theorem}

\textcolor{teal}{Theorems} \ref{Corollary: sufficient for p-adic} and \ref{theorem: formula hdim jorge} can be used to easily show that large fractal subgroups of~$W_p$ are $T_p$-rigid:
\begin{proof}[Proof of \cref{Corollary: hdim and rigidity}]
    Assume that 
    $$|G:\mathrm{St}_G(2)|=p^2.$$
    By \cref{Corollary: sufficient for p-adic}, it is enough to show that $\mathrm{hdim}_{W_p}(\overline{G})\le 1/p$ under this assumption. Our first assumption yields
    $$s_1(G):=p\log_p|G:\mathrm{St}_G(1)|-\log_p|\mathrm{St}_G(1):\mathrm{St}_G(2)|=p-1.$$
    Therefore, by \cref{theorem: formula hdim jorge}, we get
    \begin{align*}
        \mathrm{hdim}_{W_p}(\overline{G})&=1-S_{G}(1/p)=1-\sum_{n\ge 1}\frac{s_n(G)}{p^n}\le 1-\frac{s_1(G)}{p}=\frac{1}{p}.\qedhere
    \end{align*}
    \end{proof}

    \begin{remark}
    \label{remark: tfg}
        The lower bound in \cref{Corollary: hdim and rigidity} is sharp. Let us define the sequence $\mathcal{S}:=\{S_n\}_{n\ge 0}$, where
        \begin{enumerate}[\normalfont(i)]
            \item $S_0:=\langle \sigma\rangle \le \mathrm{Sym}(p)$;
            \item $\psi(S_1):=D_p(S_0)$, i.e. the diagonal embedding of $S_0$ in $S_{0}\times\overset{p}{\ldots}\times S_{0}$ via $\psi$;
            \item $\psi(S_n):=S_{n-1}\times\overset{p}{\ldots}\times S_{n-1}$ for $n\ge 2$.
        \end{enumerate}
        Let us consider the group $G_\mathcal{S}\le W_p$ (see \cite[Section 4]{RestrictedSpectra} for the definition of this family of groups $G_\mathcal{S}$). It is shown in \cite[Section 4]{RestrictedSpectra} that $G_\mathcal{S}$ is fractal and branch and
        $$\mathrm{hdim}_{W_p}(G_\mathcal{S})=1-\sum_{n\ge 1}\frac{s_n(G_\mathcal{S})}{p^n}= 1-\frac{s_1(G_\mathcal{S})}{p}=\frac{1}{p}.$$
        However, by \cref{Corollary: sufficient for p-adic}, the group $G_\mathcal{S}$ is not $T_p$-rigid as $G_\mathcal{S}/\mathrm{St}_{G_\mathcal{S}}(2)\cong C_p\times C_p$.

   On the other hand, if a fractal weakly branch group $G\le W_p$ is finitely generated (or $\overline{G}$ is topologically finitely generated) then the assumption on the Hausdorff dimension of $\overline{G}$ in \cref{Corollary: hdim and rigidity} can be relaxed to
        $$\mathrm{hdim}_{W_p}(\overline{G})\ge\frac{1}{p}.$$
        In fact, as the sequence $\{s_n(G)\}_{n\ge 1}$ is non-negative by \cref{theorem: formula hdim jorge}, both 
        $$\mathrm{hdim}_{W_p}(\overline{G})=\frac{1}{p}\quad\text{and}\quad|G:\mathrm{St}_G(2)|=p^2$$
        are only achieved at the same time if and only if $s_n(G)=0$ for $n\ge 2$ by \cref{theorem: formula hdim jorge}. However, this implies by \cite[Theorem 3.4]{RestrictedSpectra} and \cite[Theorem 3]{SunicHausdorff} (see also \cite[Proposition 7.5]{GrigorchukFinite}) that $\overline{G}$ is a group of finite type of depth 2. By \cite[Corollary~8]{Bondarenko}, the group $\overline{G}$ is not topologically finitely generated, as the set of defining patterns $G/\mathrm{St}_G(2)$ is abelian (note that any finite group of order $p^2$ is abelian).
    \end{remark}

For $T_p$-rigid groups, the Hausdorff dimension in $W_p$ is a group-invariant (as the Hausdorff dimension in $\mathrm{Aut}~T_p$, and thus in $W_p$, is preserved by conjugation in $\mathrm{Aut}~T_p$). Therefore, the Hausdorff dimension can be used to easily prove that two weakly branch groups are not isomorphic as abstract groups. We illustrate this with an example.

Let $B$ be the \textit{Basilica group}, i.e. the group $B:=\langle a,b\rangle\le \mathrm{Aut}~T_2$ (introduced by Grigorchuk and $\dot{\mathrm{Z}}$uk in \cite{Basilica}), where $\psi(a)=(1,b)$ and $\psi(b)=(1,a)\sigma$ for $\sigma=(1\,2)\in \mathrm{Sym}(2)$. Note that $B$ can also be defined as the iterated monodromy group of the complex polynomial $z^2-1$.

Let $H\le \mathrm{Aut}~T_2$ be the \textit{Brunner-Sidki-Vieira} group (introduced in \cite{BSV}) generated by the automorphisms $c$ and $d$, where $\psi(c)=(1,c)\sigma$ and $\psi(d)=(1,d^{-1})\sigma$ again for $\sigma=(1\,2)\in \mathrm{Sym}(2)$.

Both the Basilica group $B$ and the Brunner-Sidki-Vieira group are weakly branch; see \cite{BartholdiNekra, BSV}.

\begin{corollary}
\label{corollary: basilica bsv}
    The Basilica group $B$ and the Brunner-Sidki-Vieira group $H$ are not isomorphic as abstract groups.
\end{corollary}
\begin{proof}
    By either \cref{Corollary: sufficient for p-adic} or \cref{Corollary: hdim and rigidity}, the Basilica group $B$ is $T_2$-rigid. Thus, if~$B$ and $H$ are isomorphic, they are conjugate in $\mathrm{Aut}~T_2$. Since conjugation in $\mathrm{Aut}~T_2$ preserves the Hausdorff dimension, the closures of both groups should have the same Hausdorff dimension if the groups are isomorphic. However, by \cite[Examples 2.4.5 and 2.4.6]{BartholdiTree} (see \cite[Theorem~C]{pBasilica} and \cite[Theorem~A]{JorgeMikel} for the proofs), we have
    \begin{align*}
        \mathrm{hdim}(\overline{B})&=\frac{2}{3}\ne \frac{1}{3}=\mathrm{hdim}(\overline{H}).\qedhere
    \end{align*}
\end{proof}

\section{Disproval of Boston's conjecture}
\label{section: Boston conjecture}

In this last section, we prove \cref{Theorem: Boston conjecture}. The proof will use essentially all the theory developed in the previous sections. In \cref{section: p-adic}, we developed a strategy, based on \cref{lemma: orbits of elements determine basal}, to check the condition $\mathcal{F}_T=\mathcal{S}_\rho$ in \cref{proposition: characterization of rigidity}. We successfully applied this strategy in the previous section to fractal groups of $p$-adic automorphisms.  In this section, we apply this strategy to the zero-dimensional just-infinite branch groups introduced by the author in \cite{RestrictedSpectra}. The main goal is to prove that, even if the branch actions of these groups on the $p$-adic tree $T_p$ are not $T_p$-rigid, there exist natural trees $T_n$ obtained from $T_p$ by deletion of levels, such that the induced branch actions on these trees $T_n$ are indeed $T_n$-rigid. We use $T_n$-rigidity to prove that every possible branch action of the groups $G_n$ is zero-dimensional, proving \cref{Theorem: Boston conjecture} and finally disproving Boston's conjecture.

\subsection{The just-infinite groups $G_n$}

We introduce the just-infinite branch groups of $p$-adic automorphisms with zero-dimensional closures introduced by the author in \cite{RestrictedSpectra}, for some prime $p$. Let us consider $\sigma:=(1\,2\,\dotsb \,p)\in\mathrm{Sym}(p)$ and the group of $p$-adic automorphisms $W_p$.

Let us define the groups $G_n\le W_p$ as in \cite[Section 6]{RestrictedSpectra}. We present the construction for $p\ge 5$, but, as remarked at the end of \cite[Section 6]{RestrictedSpectra}, the construction can be adapted for $p=2,3$.

Let $\{l_n\}_{n\ge 1}$ be the increasing sequence of levels of $T_p$ given by $l_1=2$ and $l_{n+1}:= p^{l_n-1}$ for $n\ge 1$. Given an automorphism $g\in W_p$, we set $d_0(g):=g$ and, for every $i\ge 1$, we define the automorphism $d_i(g)\in W_p$ via
$$\psi_i(d_i(g))=(g,\dotsc,g).$$
We define $\psi(a):=(1,\dotsc,1)\sigma\in W_p$. For all $n\ge 1$, we further define $b_n$ via
$$\psi_{l_{n}}(b_{n})=(d_0(a),d_1(a),\dotsc,d_{l_{n+1}-1}(a),1,\dotsc,1,b_{n+1}).$$
Lastly, for every $n\ge 1$, we define the finite group 
$$A_n:=\langle d_i(a)~|~i=0,1,\dotsc, l_n-1\rangle$$
and the infinite group 
$$G_n:=\langle A_n,b_{n}\rangle.$$

In the following proposition, we summarize the main properties of the groups~$G_n$ proved by the author in \cite{RestrictedSpectra}:

\begin{proposition}
    \label{proposition: properties of Gn}
    For every $n\ge 1$, the subgroup $G_n\le W_p$ satisfies the following:
    \begin{enumerate}[\normalfont(i)]
        \item we have $\mathrm{St}_{G_n}(l_n)=\langle b_n\rangle^{A_n}$ and $G_n=\mathrm{St}_{G_n}(l_n)\rtimes A_n$ \cite[Proposition 6.6(i)]{RestrictedSpectra};
        \item we have $\varphi_v(\mathrm{St}_{G_n}(l_n))=G_{n+1}$ for any $v$ at level $l_n$ \cite[Proposition 6.6(ii)]{RestrictedSpectra};
        \item the group $G_n$ is just-infinite and branch \cite[Proposition 6.11]{RestrictedSpectra};
        \item the closure of $G_n$ in $W_p$ has Hausdorff dimension zero \cite[Proposition 6.12]{RestrictedSpectra}.
    \end{enumerate}
\end{proposition}

Lastly, for each $n\ge 1$, we define the collection of levels $\{t_k^n\}_{k\ge 0}$ as $t_0^n:=0$ and $t_k^n:=\sum_{i=0}^{k-1} l_{n+i}$ for each $k\ge 1$. We also define the auxiliary finite groups 
$$\widetilde{A}_n:=\langle d_i(a)~|~i=0,1,\dotsc, l_n-2\rangle.$$
We note that both $A_n$ and $\widetilde{A}_n$ are elementary $p$-abelian by \cite[Lemma~6.4]{RestrictedSpectra}. Furthermore, we have the relations
\begin{align}
    \label{align: properties of An}
    A_{n+1}\cong C_p^{l_{n+1}}=C_p^{p^{l_n-1}}= C_p^{|\widetilde{A}_n|}\quad  \text{and}\quad \psi(A_n)=D_p(\widetilde{A}_n)\rtimes \langle a\rangle,
\end{align}
where $D_p(\widetilde{A}_n)$ denotes the diagonal embedding of $\widetilde{A}_n$ in $A_n$ via $\psi$.

\subsection{Deletion of levels}

Given a spherically homogeneous rooted tree $T$, we may obtain a new tree $\widetilde{T}$ from $T$ by \textit{deletion of some collection of levels} $\{a_n\}_{n\ge 1}\subseteq \mathbb{N}$. The set of vertices of $\widetilde{T}$ are precisely the vertices in $T$ not at levels $\{a_n\}_{n\ge 1}$ and the set of edges is the induced one from $T$, i.e. the unique possible set of edges such that $\widetilde{T}$ is a tree and two distinct vertices in $\widetilde{T}$ are comparable if and only if they are comparable in $T$. Note that $\widetilde{T}$ is also a spherically homogeneous rooted tree. Clearly, if $\rho:G\to \mathrm{Aut}~T$ is a (weakly) branch action, we obtain an induced (weakly) branch action $\widetilde{\rho}:G\to \mathrm{Aut}~\widetilde{T}$ for any tree $\widetilde{T}$ obtained from $T$ by deletion of levels.

Let $T_p$ be the $p$-adic tree. For each $n\ge 1$, we define the tree $T_n$ by deleting from~$T_p$ all the levels but the ones in the collection $\{t^n_k\}_{k\ge 0}$. In other words, for each $k\ge 0$, the level $k$ in $T_n$ corresponds to the level $t_k^n$ in $T_p$.

We will be working simultaneously with the branch action of $G_n$ on $T_p$ given in the definition of $G_n$ as a group of $p$-adic automorphisms and with its induced action on $T_n$. We shall write 
$$\rho_n:G_n\to W_p\quad \text{and}\quad \tau_n:G_n\to \mathrm{Aut}~T_n$$
for the branch action of $G_n$ on $T_p$ and the induced one (by $\rho_n$) on $T_n$, respectively. Note that
$$\mathrm{St}_{\tau_n}(k)=\mathrm{St}_{\rho_n}(t_k^n)\quad \text{and}\quad \mathrm{st}_{\tau_n}(v)=\mathrm{st}_{\rho_n}(v)$$
for any $k\ge 0$ and any vertex $v$ at some level $t_k^n$.

\subsection{$T_n$-rigidity}

Before proceeding with the proofs, we set a last piece of notation. For a branch action $\chi:G\to \mathrm{Aut}~T$, a vertex $v\in T$ and a subgroup $H\le G$, we write 
$$H_v:=\varphi_v(\chi(H\cap \mathrm{st}_{\chi}(v)))\le \mathrm{Aut}~T_v,$$
where $T_v$ is the subtree rooted at $v$.

We first prove the following technical lemma:

\begin{lemma}
    \label{lemma: indices of action of H}
    For each $n\ge 1$, let $\mathrm{st}_{\rho_n}(w)\le H\le \mathrm{St}_{\rho_n}(1)$, where $w$ lies at level $t_2^n$ in $T_p$. Assume further that
   $$\log_p |G_n:H\mathrm{St}_{\rho_n}(t_1^n)|=\ell.$$
   Then 
   $$\log_p|\mathrm{St}_{\rho_n}(t_1^n):H\cap \mathrm{St}_{\rho_n}(t_1^n)|=k\cdot p^{l_n-\ell}$$
   for some integer $k\ge 0$.
\end{lemma}
\begin{proof}
    As $H\le \mathrm{St}_{\rho_n}(1)$, in the remainder of the proof we work with $H_v$, where $v$ is the unique vertex at level 1 in~$T_p$ above $w$. To simplify notation we shall write $H$, $\mathrm{St}_{\rho_n}(t_1^n)$, etc for $H_v$, $\mathrm{St}_{\rho_n}(t_1^n)_v$, etc, as we only work with these projections from now on in the proof.

    First, we may define $H_1$ and $H_2$ as
    $$\frac{H\mathrm{St}_{\rho_n}(t_1^n)}{\mathrm{St}_{\rho_n}(t_1^n)}=H_1\le \widetilde{A}_n\quad\text{and}\quad \frac{H\cap \mathrm{St}_{\rho_n}(t_1^n)}{\mathrm{St}_{\rho_n}(t_2^n)}=H_2\le \frac{\mathrm{St}_{\rho_n}(t_1^n)}{\mathrm{St}_{\rho_n}(t_2^n)},$$
    and note that the definition of $H_2$ makes sense as $H\ge \mathrm{st}_{\rho_n}(w)\ge \mathrm{St}_{\rho_n}(t_2^n)$. Then, we define $$K:=H_2\rtimes H_1\le \frac{H\mathrm{St}_{\rho_n}(t_1^n)}{\mathrm{St}_{\rho_n}(t_2^n)},$$
    which is well defined. Indeed, on the one hand, by \cref{proposition: properties of Gn}\textcolor{teal}{(i)} both $H_1$ and~$H_2$ can be realized as well-defined subgroups of $H\mathrm{St}_{\rho_n}(t_1^n)/\mathrm{St}_{\rho_n}(t_2^n)$ with trivial intersection. On the other hand, the subgroup $H_2$ is normalized by $H_1$, as $H\cap \mathrm{St}_{\rho_n}(t_1^n)$ is normalized by $H$ and $\mathrm{St}_{\rho_n}(t_1^n)/\mathrm{St}_{\rho_n}(t_2^n)$ is abelian (so the sections of $H$ and $H_1$ at level $t_1^n$ do not play a role modulo the stabilizer of level $t_2^n$).

    Let $0\le i\le p-1$ be such the projection of $b_n^{a^i}$ at $v$ is non-trivial and let us write~$\widetilde{b}_n$ for this projection. Then
    $$\psi_{l_n-1}(\widetilde{b}_n)=(d_0(a),d_1(a),\dotsc, d_{l_{n+1}-1}(a)).$$
    A conjugate $\widetilde{b}_n^h$ with $h\in \widetilde{A}_n$ corresponds (via $\psi_{l_n-1}$) to a permutation of the tuple $$(d_0(a),d_1(a),\dotsc, d_{l_{n+1}-1}(a)).$$
    Note that $d_i(a)$ commutes with $d_j(a)$ for any $i,j\ge 0$, so $\langle \widetilde{b}_n\rangle^{\widetilde{A}_n}$ is elementary $p$-abelian. Since this action by conjugation of $\widetilde{A}_n$ permuting the components of this tuple is transitive and there are exactly $l_{n+1}=|\widetilde{A}_n|$ components in this tuple, for each $0\le i\le l_{n+1}-1$, there is a unique element $h\in \widetilde{A}_n$ such that $d_i(a)$ is at the first component of $\widetilde{b}_n^h$. Thus, for each $h\in \widetilde{A}_n$, the conjugate $\widetilde{b}_n^h$ is not contained in the subgroup $\langle \widetilde{b}_n\rangle^{\widetilde{A}_n\setminus\{h\}}$. Therefore, the elementary $p$-abelian group $\langle \widetilde{b}_n\rangle^{\widetilde{A}_n}$ is generated by exactly $|\widetilde{A}_n|$ distinct elements of order $p$, corresponding to the distinct conjugates of $\widetilde{b}_n$. This, together with \cref{align: properties of An} and \cref{proposition: properties of Gn}\textcolor{teal}{(i)}, yields
    \begin{align}
        \label{align: the quotient is An}
        \frac{\mathrm{St}_{\rho_n}(t_1^n)}{\mathrm{St}_{\rho_n}(t_2^n)}\cong\langle \widetilde{b}_n \rangle^{\widetilde{A}_n}\cong C_p^{|\widetilde{A}_n|}\cong A_{n+1}.
    \end{align}
    Now note that if $\widetilde{b}_n^h\in H_2\le \mathrm{St}_{\rho_n}(t_1^n)/\mathrm{St}_{\rho_n}(t_2^n)$ for some $h\in \widetilde{A}_n$, then $\widetilde{b}_n^{hH_1}\subseteq H_2$ as $H_2^{H_1}=H_2$. Hence, we obtain
    \begin{align*}
        H_2&\cong\prod_{s\in S}\langle \widetilde{b}_n^{s}\rangle^{H_1}\cong C_p^{\# S\cdot |H_1|}
    \end{align*}
    for some subset $S$ of left cosets of $H_1$ in $\widetilde{A}_n$. Therefore
    \begin{align*}
       \log_p|\mathrm{St}_{\rho_n}(t_1^n):H\cap \mathrm{St}_{\rho_n}(t_1^n)|&=\log_p|A_{n+1}|-\log_p|H_2|= |\widetilde{A}_{n}|-\#S\cdot |H_1|\\
       &=k\cdot |H_1|= k\cdot p^{l_n-\ell}, 
    \end{align*}
    where $k:=|\widetilde{A}_n:H_1|-\#S\in \mathbb{N}\cup \{0\}$ and the last equality follows from the assumption in the statement.
\end{proof}

\cref{lemma: indices of action of H} yields a dichotomy for $H$: either $H$ contains the $t_1^n$ level stabilizer or the  logarithmic indices in the statement of \cref{lemma: indices of action of H} cannot be both small at the same time. This is the key to prove our main lemma:

\begin{lemma}
    \label{lemma: climb levels}
    For each $n\ge 1$, let $H\le G_n$ be such that $|G_n:H|\le p^{l_n}$ and $H\ge \mathrm{st}_{\rho_n}(w)$ for some vertex $w$ at level $t_2^n$ in $T_p$. Then $H\ge \mathrm{St}_{\rho_n}(t_1^n)$. 
\end{lemma}
\begin{proof}
    We shall work modulo $\mathrm{St}_{\rho_n}(t_2^n)$ in the following without stating it explicitly to simplify notation. First, note that from \cref{align: properties of An} and \cref{proposition: properties of Gn} we get
    \begin{align*}
        \psi(\mathrm{St}_{\rho_n}(t_1^n))&=\psi(\langle b_n\rangle^{A_n})= \langle (\widetilde{b}_n,1,\dotsc,1)\rangle^{D_p(\widetilde{A}_n)\rtimes \langle a\rangle}\\
        &=\prod_{i=0}^{p-1}\big(\langle \widetilde{b}_n\rangle^{\widetilde{A}_n} \big)^{a^i} = \prod_{i=1}^p \varphi_i(\mathrm{St}_{\rho_n}(t_1^n)),
    \end{align*}
    where $\varphi_i$ denotes the projection at vertex $i$ after naming $1,\dotsc, p$ the vertices at level 1. Recall that by \cref{align: the quotient is An}, we also have 
    $$\varphi_i(\mathrm{St}_{\rho_n}(t_1^n))=\langle \widetilde{b}_n\rangle^{\widetilde{A}_n}\cong A_{n+1}.$$

    Let $i\in T_p$ be the unique vertex at level 1 in $T_p$ above $w$. Now, clearly we have $\mathrm{st}_{\rho_n}(w)\ge\ker \varphi_i|_{\mathrm{St}_{\rho_n}(t_1^n)}$, and we shall prove the reverse containment. An element $g\in \varphi_i(\mathrm{St}_{\rho_n}(t_1^n))\cong A_{n+1}$ fixes $w$ if and only if $g=1$, as every non-trivial element of~$A_{n+1}$ moves every vertex at level $t_1^{n+1}$, so, in particular, it moves $w$. Thus $\mathrm{st}_{\rho_n}(w)\le \ker \varphi_i|_{\mathrm{St}_{\rho_n}(t_1^n)}$, as $\mathrm{st}_{\rho_n}(w)\le \mathrm{St}_{\rho_n}(t_1^n)$. Then $\mathrm{st}_{\rho_n}(w)=\ker \varphi_i|_{\mathrm{St}_{\rho_n}(t_1^n)}$, which yields
    \begin{align}
        \label{align: description of st}
        \mathrm{st}_{\rho_n}(w)=\ker \varphi_i|_{\mathrm{St}_{\rho_n}(t_1^n)}=\prod_{j\ne i}\varphi_j(\mathrm{St}_{\rho_n}(t_1^n))\cong \prod_{i=1}^{p-1}A_{n+1},
    \end{align}
    where $j$ ranges over all the vertices at level 1 in $T_p$ but $i$.

    We claim that if $a\in H$, then $H\ge \mathrm{St}_{\rho_n}(t_1^n)$. Indeed, if $a\in H$, then $H$ acts transitively on the first level of $T_p$. As $H\ge \mathrm{st}_{\rho_n}(w)$, by the description of $\mathrm{st}_{\rho_n}(w)$ in \cref{align: description of st}, we obtain further that, for $v$ the unique vertex at level $t_1^n$ in $T_p$ above $w$, the subgroup $H\cap \mathrm{St}_{\rho_n}(t_1^n)$ acts as $A_{n+1}$ (in particular, transitively) on the descendants of $v$ at level $t_2^n$ in $T_p$, since $a\in H$. Thus, by \cref{lemma: orbits of elements determine basal} we get
    $$H\ge \mathrm{st}_{\rho_n}(v)=\mathrm{St}_{\rho_n}(t_1^n).$$
    Therefore, let us assume in the following that $a\notin H$, which implies $H\le \mathrm{St}_{\rho_n}(1)$. Then, we may apply \cref{lemma: indices of action of H} to obtain that, if 
   $$\log_p |G_n:H\mathrm{St}_{\rho_n}(t_1^n)|=\ell,$$
   then 
   $$\log_p|\mathrm{St}_{\rho_n}(t_1^n):H\cap \mathrm{St}_{\rho_n}(t_1^n)|=k\cdot p^{l_n-\ell}$$
   for some integer $k\ge 0$. Hence, we may rewrite the equation
   \begin{align*}
    \log_p|G_n:H|&=\log_p|G_n:H\mathrm{St}_{\rho_n}(t_1^n)|+ \log_p|H\mathrm{St}_{\rho_n}(t_1^n):H|\\
    &=\log_p |G_n:H\mathrm{St}_{\rho_n}(t_1^n)|+\log_p|\mathrm{St}_{\rho_n}(t_1^n):H\cap \mathrm{St}_{\rho_n}(t_1^n)|,
   \end{align*}
   as
   \begin{align}
       \label{align: equation on orders}
       \log_p|G_n:H|=\ell + k\cdot p^{l_n-\ell}.
   \end{align}
    Note that $k=0$ implies $H\ge \mathrm{St}_{\rho_n}(t_1^n)$ yielding the result. Similarly, if $\ell=l_n$, we get $H=\mathrm{St}_{\rho_n}(t_1^n)$. Thus, as we are dealing with $p$-groups, $\ell$ is also an integer, so it only remains to show that \cref{align: equation on orders} does not have a solution for integers $k,\ell$ such that $1\le k$ and $0\le \ell\le  l_n-1$. As $p\ge 2$, we have
   $$\frac{d}{dk}(\ell+k\cdot p^{l_n-\ell})=p^{l_n-\ell}>0\quad\text{and}\quad \frac{d}{d\ell}(\ell +p^{l_n-\ell})=1-p^{l_n-\ell}\log_e p\le 1- 2\log_e 2<0$$
   for any integers $1\le k$ and $0\le \ell\le l_n-1$. Therefore, the function  $\ell+k\cdot p^{l_n-\ell}$ achieves its minimum at $k=1$ and $\ell=l_n-1$, which is given by
   $l_n-1+p$. As $p\ge 2$, we get the strict inequality
   $$\log_p|G_n:H|\le l_n<l_n-1+p,$$
   which shows that \cref{align: equation on orders} has no solution for integers $0\le \ell\le l_n-1$ and $k\ge 1$.
\end{proof}

Now, if $u\in T_n$, we denote by $|u|$ the level in $T_n$ at which $u$ lies. Note that we may identify $(T_n)_u$ with $T_{n+|u|}$ by the choice of the sequences $\{t_k^n\}_{k\ge 0}$. 

Finally, we can use \cref{lemma: climb levels} to prove $T_n$-rigidity of the group $G_n$:

\begin{theorem}
\label{theorem: my groups are T-rigid}
    For each $n\ge 1$, the group $G_n$ is $T_n$-rigid.
\end{theorem}
\begin{proof}
Let us fix $n\ge 1$.  We only need to show that if $\mathrm{st}_{\tau_n}(w_k)\le H\le G_n$ for some $w_k\in T_n$ and $|G_n:H|=p^{l_n}$, then $H=\mathrm{St}_{\tau_n}(1)=\mathrm{st}_{\tau_n}(v)$, where the last equality always holds for any $v$ at level 1 in $T_n$. Then, as by \cref{proposition: properties of Gn}\textcolor{teal}{(ii)} we have  
    $$\varphi_{v}(\mathrm{st}_{\tau_n}(v))=G_{n+1},$$
    we may argue as in the proof of \cref{theorem: sufficient condition for rigidity in p-adic} to obtain $\mathcal{F}_{T_n}=\mathcal{S}_{\tau_n}$. Hence, the result follows from \cref{proposition: characterization of rigidity}. 
    
    Let us consider $H_u=\varphi_u(\tau_n(H\cap\mathrm{st}_{\tau_n}(u)))$ for the unique vertex $u$ two levels above $w_k$ in~$T_n$. If $H_u\ge \mathrm{st}_{\tau_{n+|u|}}(w)$ for some $w\in (T_n)_{u}=T_{n+|u|}$ such that $uw$ is above $w_k$ in $T_n$, then $H\ge \mathrm{st}_{\tau_n}(uw)$. In fact, this follows from \cref{lemma: level-transitivity}, as $H$ is completely determined by the $H$-orbit of $w_k$ by \cref{lemma: orbits of elements determine basal}.

    Now, again by \cref{proposition: properties of Gn}\textcolor{teal}{(ii)}, we have
    $$p^{l_n}=|G_n:H|\ge |(G_{n})_u:H_u|=|G_{n+|u|}:H_u|.$$
    As 
    $$p^{l_{n+|u|}}\ge p^{l_n}\ge |G_{n+|u|}:H_u|,$$
    the subgroup $H_u$ satisfies the assumptions in \cref{lemma: climb levels}.  Thus, we may apply \cref{lemma: climb levels} to $H_u$ to obtain
    $$H_u\ge \mathrm{St}_{\tau_{n+|u|}}(1)=\mathrm{st}_{\tau_{n+|u|}}(w)$$
    for any $w$ at level 1 in $T_{n+|u|}$. Therefore, by the above observation
    $$H\ge \mathrm{st}_{\tau_n}(w_{k-1}),$$
    where $w_{k-1}:=uw$ is the unique vertex one level above $w_k$ in $T_n$. This argument may be iterated (as $l_{n+k}\ge l_n$ for every $k\ge 0$) yielding
    $$H\ge \mathrm{st}_{\tau_n}(w_1),$$
    where $w_1$ is at level 1 in $T_n$. However, as
    $$|G_n:H|=p^{l_n}=|G_n:\mathrm{st}_{\tau_n}(w_1)|,$$
    we get the equality of subgroups $H= \mathrm{st}_{\tau_n}(w_1)$ as wanted.
\end{proof}

\begin{figure}[H]

        \centering
 \includegraphics{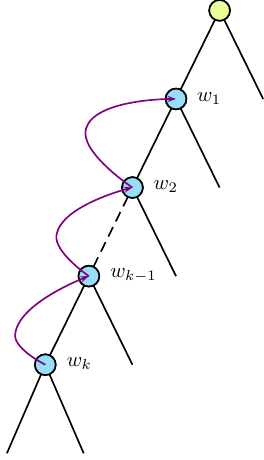}
        \caption{An illustration of the proof of \cref{theorem: my groups are T-rigid}. At each step, we apply \cref{lemma: climb levels} to climb one level up in $T_n$ and obtain a larger vertex stabilizer inside the subgroup $H$.}
        \label{figure: not m cousins}
\end{figure}

\subsection{Disproval of Boston's conjecture}

As remarked in \cite{RestrictedSpectra}, to prove \cref{Theorem: Boston conjecture}, it is enough to show that every branch action of $G_n$ on $T_p$ (as a subgroup of $W_p$) is zero-dimensional. We conclude the paper by using the $T_n$-rigidity of the group~$G_n$ to prove this:

\begin{proof}[Proof of \cref{Theorem: Boston conjecture}]
    For each $n\ge 1$, the group $G_n\le W_p$ is $T_n$-rigid, by \cref{theorem: my groups are T-rigid}. Let us consider any branch action on the $p$-adic tree $\chi_n:G_n\to \mathrm{Aut}~T_p$, such that $\chi_n(G_n)\le W_p$. We obtain an induced branch action $\widetilde{\chi}_n:G_n\to \mathrm{Aut}~T_n$ by deletion of levels. However, as $G_n$ is $T_n$-rigid, the branch actions $\tau_n$ and $\widetilde{\chi}_n$ must be conjugate in $\mathrm{Aut}~T_n$. In particular
    \begin{align}
    \label{align: equality of log ind}
        \log_p|G_n:\mathrm{St}_{\rho_n}(t_k^n)|&=\log_p|G_n:\mathrm{St}_{\tau_n}(k)|=\log_p|\widetilde{\chi}_n(G_n):\mathrm{St}_{\widetilde{\chi}_n}(k)|\\
        &=\log_p|\chi_n(G_n):\mathrm{St}_{\chi_n}(t_k^n)|\nonumber
    \end{align}
    for every $k\ge 1$. In \cite[Proposition 6.12]{RestrictedSpectra}, the author showed that
    $$\mathrm{hdim}_{W_p}(\overline{G}_n)=\liminf_{k\to \infty}\frac{\log_p|G_n:\mathrm{St}_{\rho_n}(t_k^n)|}{\log_p|W_p:\mathrm{St}_{W_p}(t_k^n)|}=0.$$
    Thus, combining this with \cref{align: equality of log ind} yields
    \begin{align*}
        \mathrm{hdim}_{W_p}(\overline{\chi_n(G_n)})&=\liminf_{k\to \infty}\frac{\log_p|\chi_n(G_n):\mathrm{St}_{\chi_n}(k)|}{\log_p|W_p:\mathrm{St}_{W_p}(k)|}\\
        &\le \liminf_{k\to \infty}\frac{\log_p|\chi_n(G_n):\mathrm{St}_{\chi_n}(t_k^n)|}{\log_p|W_p:\mathrm{St}_{W_p}(t_k^n)|}\\
    &=\liminf_{k\to \infty}\frac{\log_p|G_n:\mathrm{St}_{\rho_n}(t_k^n)|}{\log_p|W_p:\mathrm{St}_{W_p}(t_k^n)|}=0.
    \end{align*}
    As $\chi_n$ was an arbitrary branch action of $G_n$ on $T_p$ (such that $\chi_n(G_n)\le W_p$), each~$G_n$ is a counterexample to Boston's conjecture.
\end{proof}

We note that the proof above yields that the Hausdorff dimension of $\overline{\chi_n(G_n)}$ in $\mathrm{Aut}~T_p$ is also zero even when $\chi_n(G_n)$ is not a subgroup of $W_p$.



\bibliographystyle{unsrt}

\end{document}